\documentclass[11pt]{amsart}

\usepackage{amsmath,amssymb}
\usepackage{amsthm,verbatim,amsfonts,amscd,graphicx}
\usepackage{xcolor}
\usepackage{graphics}

\usepackage{euscript,enumerate}
\usepackage{url,hyperref}

\newcommand{\R}{{\mathbb R}}

\newcommand{\p}{\partial}
\newcommand{\fr}{\frac}
\newcommand{\la}{\langle}
\newcommand{\ra}{\rangle}

\newcommand{\e}{\epsilon}
\newcommand{\ka}{\kappa}

\newcommand{\be}{\begin{equation}}
\newcommand{\ba}{\begin{aligned}}
\newcommand{\bee}{\begin{equation*}}
\newcommand{\ee}{\end{equation}}
\newcommand{\ea}{\end{aligned}}
\newcommand{\eee}{\end{equation*}}

\theoremstyle{plain}
\newtheorem{theorem}{Theorem}[section]
\newtheorem{thm}[theorem]{Theorem}
\newtheorem{corollary}[theorem]{Corollary}

\newtheorem{lemma}[theorem]{Lemma}
\newtheorem{proposition}[theorem]{Proposition}
\newtheorem{prop}[theorem]{Proposition}

\newtheorem{claim}{Claim}[section]

\theoremstyle{remark}

\newtheorem{remark}[theorem]{Remark}

\theoremstyle{definition}
\newtheorem{definition}[theorem]{Definition}

\numberwithin{equation}{section}

\usepackage[utf8]{inputenc}
\usepackage{mathtools}

\title{Convergence of Curve Shortening Flow to Translating Soliton}
\author{Beomjun Choi}
\address{{\bf Beomjun Choi:} Department of Mathematics, Columbia University, 2990 Broadway, New York, \newline \hphantom \quad\, NY 10027, USA}
\email{bc2491@columbia.edu}
\author{Kyeongsu Choi}
\address{{\bf Kyeongsu Choi:} Department of Mathematics, Massachusetts Institute of Technology, \newline \hphantom \quad\,  77 Massachusetts Ave,  Cambridge,  MA 02139, USA.}
\email{choiks@mit.edu}
\author{Panagiota Daskalopoulos}
\address{{\bf Panagiota Daskalopoulos:} Department of Mathematics, Columbia University, 2990 Broadway, \newline \hphantom \quad\, New York,  NY 10027, USA}
\email{pdaskalo@math.columbia.edu}

\begin{document}

\begin{abstract}
This paper concerns with  the asymptotic behavior  of    complete non-compact convex curves    embedded in $\mathbb{R}^2$ under  the 
$\alpha$-curve shortening flow  for exponents $\alpha >  \frac 12$. We show that  any such curve having in addition  its  two ends asymptotic to two parallel lines, 
converges under  $\alpha$-curve shortening flow  to the unique  translating soliton whose ends are asymptotic to the same parallel lines. This is a new result even in the standard case $\alpha=1$,  and we prove  for all exponents  up to  the critical case $\alpha>\frac{1}{2}$.
\end{abstract}

\maketitle

\section{introduction}

Given a positive constant $\alpha$, we say that   a one-parameter family of immersions $X: N\times [0,T] \mapsto \R^2$ is a convex complete solution of the $\alpha$-curve shortening flow ($\alpha$-CSF in abbreviation) if each image $M_t \coloneqq X(N\times \{t\})$ is a smooth convex complete curve and the following holds 
\be\label{eqn-CSFG}
\fr{\p}{\p t} \, X(p,t)= \bar \ka ^\alpha(p,t) \, n(p,t)
\ee
where $\bar \ka(p,t)$ is the curvature of $M_t$ at $X(p,t)$, and  $n(p,t)$ is the unit normal vector pointing the convex hull of $M_t$.  Throughout the paper, if we need a distinction in the parametrizations of the curvature,   we use $\bar \ka = \bar \ka(p,t)$
for the parametrization as in \eqref{eqn-CSFG} and 
we use $\ka=\ka(\theta,t)$, where $\theta$ denotes the angle between $n(p,t)$ and $e_1$. 
\bigskip

In 1984 \cite{Ga}, Gage showed that the CSF ($\alpha=1$) makes closed convex curves circular. Jointly with Hamilton, he 
established  the improved result \cite{GH} that closed convex solutions of the CSF converge to circles after rescaling. Namely, closed convex solutions converge to \textit{shrinking} solitons.

\bigskip

Regarding  complete non-compact solutions,  Ecker and Huisken \cite{EH1} proved that asymptotically conical $n$-dimensional entire graphs in $\R^{n+1}$ which evolve by the mean curvature flow (a  higher dimensional analogue to the CSF)   converge to \textit{expanding} solitons after rescaling.  

\bigskip

In this paper, we study the convergence of the CSF to \textit{translating} solitons. Our main result states as follows:

\begin{theorem}\label{thm-main} Assume that $M_0$ is a strictly  convex smooth non-compact complete curve embedded in $\R^2$, and that  its two ends are asymptotic to two parallel lines. Then, for given $\alpha>\fr{1}{2}$ the unique strictly  convex complete solution of the $\alpha$-CSF converges, as  $t \to \infty$, locally smoothly to the unique translating soliton of the $\alpha$-CSF which is asymptotic to the two lines.  \end{theorem}

In the classical  case $\alpha=1$, the translating solitons are the Grim Reaper curves which are homothetic to the curve $\Gamma=\{(x_1,-\log \cos x_1): x_1 \in (-\frac{\pi}{2},\frac{\pi}{2})\}$ up to rotation. Thus, the Grim Reaper curves have two ends asymptotic to two parallel lines. 

On the other hand, by the result in \cite{CDKL} a convex complete graph $M_0$ over an open interval $I \subset \R$ (either bounded or unbounded)  remains as a convex complete graph $M_t$ over $I$ under the CSF for the all time. Therefore, the initial graph $M_0$ must be defined over a bounded interval in order to converge to a Grim Reaper curve. Namely, for the convergence to a Grim Reaper curve it is necessary to assume that the two ends of $M_0$ are asymptotic to two parallel lines.

\bigskip 

However, it was revealed by Calabi in \cite{Ca} that translating solitons to the $\frac{1}{3}$-CSF are the parabola $\Gamma=\{(x_1,x_1^2):x_1 \in \R\}$ up to affine transforms. Namely, a translating soliton to the $\frac{1}{3}$-CSF is not contained in a strip. Therefore, an initial graph $M_0$ must be an entire graph to converge to a parabola. Naturally, the two cases $\alpha=1$ and $\alpha=\frac{1}{3}$ would expect different types of proofs for the convergence to translating solutions. 
 
In this work, we concentrate on the range of exponents  $\alpha >\frac{1}{2}$, due to the result of Urbas \cite{U2} that translating solitons to the $\alpha$-Gauss curvature flow ($\alpha$-GCF) with  $\alpha>\frac{1}{2}$ are contained in cylinders while those with $0<\alpha \leq \frac{1}{2}$ are entire graphs. We recall that the GCF is also a higher dimensional analogue to the CSF.

\bigskip
 
We treat the $\alpha$-CSF with $\alpha<1$ as a fast diffusion type equation and Proposition \ref{prop-lowerbddfast}, the asymptotic property of the ends of $M_t$, follows from this consideration. Then, the condition $\alpha> \frac{1}{2}$ yields a sharp lower bound of curvature decay which is needed to prove convergence of solutions to the translating solitons.

However, we will also derive   upper bounds for the curvature and its derivatives for $\alpha>\frac{1}{3}$ which are  independent from  the shape of the ends of $M_t$.  This $\alpha=\frac{1}{3}=\frac{1}{1+2}$ is also a critical exponent which is due to the fact that in this case the equation is invariant under  affine transformations. By the work \cite{Ca} of Calabi, the shrinkers, expanders, and translators to the $\frac{1}{n+2}$-GCF are ellipsoids, hyperboloids, and paraboloids, respectively. Namely, the $\frac{1}{n+2}$-GCF has infinitely many different solitons, but they all are equivalent up to affine transformations. 

 Recently, Andrews-Guan-Ni \cite{AGN} showed the convergence of closed solutions of the $\alpha$-GCF to shrinking solitons for $\alpha>\frac{1}{n+2}$, and Brendle-Choi-Daskalopoulos \cite{BCD} obtained the uniqueness of closed shrinkers for $\alpha>\frac{1}{n+2}$. In this regard, the upper bounds for the curvature and its derivatives for $\alpha>\frac{1}{3}$ in this paper 
 could  be helpful   in studying the convergence of entire graph solutions to the translating solitons for $\frac{1}{3} <\alpha \leq \frac{1}{2}$.

\bigskip

\begin{remark}[\textit{Local convergence}]\label{remark-localconvergence}
In  Theorem \ref{thm-main}, the term "locally smoothly converges" indicates that, for instance, if the two ends for the initial curve $M_0$ are asymptotic to $\{x_1=-1\}$ and $\{x_1=1\}$ then after translating the solution  as $\{M_t-h(t)\, e_{2}\}$ so that it  contains the origin, it smoothly converges to the soliton on $[-1+\delta, 1-\delta]\times \R$,  for every small $\delta>0$. For more details, see the theorem \ref{thm-goal}.
\end{remark}

\bigskip

\begin{remark}[\textit{Translating solitons of $\alpha>1$ contain flat lines}]
Given $\alpha>1$, the $C^1$ convex translating solitons have two half lines and the solitons are not of $C^{\infty}$ class. See \cite{U2}. For example, given $\alpha>1$ there exists a convex even function $f:[-1,1]\to \mathbb{R}$ such that
\begin{itemize}
\item $f$ is smooth strictly convex on $(-1,1)$, and $|Df|(x_1)\to +\infty$ as $|x_1| \to 1$,
\item $\Gamma\coloneqq \{(\pm 1,x_2):x_2 \geq f(1)\}\cup \{(x_1,f(x_1)):|x_1| <1\}$ is a translating soliton to the $\alpha$-CSF.
\end{itemize} 
In the higher dimensional case of the  GCF,  the evolution of surfaces with flat sides has  been studied as a free-boundary problem which is also motivated from the wearing precess of  stones  \cite{Fi, Ha0, DH2, DL}.  
In particular the works \cite{DH2, DL}, treat the  GCF  as a slow diffusion of  a similar nature as that appearing in the  the Porous medium equation.  Similarly, the $\alpha$-CSF with  $\alpha>1$ sufficiently large is  a slow diffusion equation. This  can be seen from the  evolution equation  the  speed $\kappa^\alpha$ which given in  \eqref{eq-k^a}.  
Thus, in this case too one may consider weakly convex initial data with flat lines and study its evolution. 
However, in this work  we consider only strictly  convex  and  complete initial data  and we  show
that the solution converges to a weakly convex $C^1$ translator with flat lines.

In addition, it was recently discovered in \cite{CDL} that   translating solitons   to the GCF in $\R^3$ have flat sides if their asymptotic cylinders at infinity have flat sides. Namely,  translating solitons to  nonlinear flows may have flat sides, arising from slow  diffusion  at   infinity.
\end{remark}

\bigskip

\noindent\textit{Discussion on the Proof}: The key idea of the  paper is to utilize the monotonicity of the functional

\[\ba J(t) =\fr{(\alpha+1)^2}{\alpha^2}\int (\ka^\alpha)^2_\theta - (\ka^\alpha)^2d \theta.\ea\] 
Such a functional was used in \cite{DHS} for the classification of closed convex ancient solutions to the CSF.
Note that on a  closed convex solution the function  $\kappa$ is $2\pi$-periodic and one  can simply obtain $\p_t J \leq 0$ by  integration by parts. However, in our non-compact case  boundary terms appear after we  integrate by parts (see  Proposition \ref{prop-basic}). Heuristically, we have
\be\label{eq-formal}\ba \p_t J 
%&=\fr{(\alpha+1)^2}{\alpha^2} \int _0^\pi 2(\ka^\alpha)_\theta(\ka^\alpha)_{t\theta} -2(\ka^\alpha)(\ka^\alpha)_td\theta 
%\\&= -\fr{2(\alpha+1)^2}{\alpha^2}\int^{\pi}_0 \fr{\ka_t (\ka^\alpha)_t }{\ka^2} d\theta + \Big (\fr{2(\alpha+1)^2}{\alpha^2} (\ka^\alpha)_\theta (\ka^\alpha)_t \Big )^{\theta=\pi} _{\theta=0}\\
%& = \int_0^\pi \fr{-2(\alpha+1)^2}{\alpha} \fr{\ka_t^2}{\ka^{3-\alpha}}d\theta+ \Big (\fr{2(\alpha+1)^2}{\alpha^2} (\ka^\alpha)_\theta (\ka^\alpha)_t \Big )^{\theta=\pi} _{\theta=0}\\& 
= \fr{2(\alpha+1)^2}{\alpha^2}\Big (-\alpha\int_0^\pi \ka^{\alpha+1} [(\ka^\alpha)_{\theta\theta}- \ka^\alpha]^2d\theta+ \Big ( (\ka^\alpha)_\theta (\ka^\alpha)_t \Big )^{\theta=\pi} _{\theta=0} \Big ).
\ea\ee

\medskip

The most challenging part of our prove is to show that the  boundary terms vanish. For that  it is crucial to   derive   local derivative estimates on  the speed $\ka^\alpha$ in (see Section 3). We then combine these  estimates   which we then combined  with  our H\"older estimate for  $\ka^\alpha$ (see  in Section 2).  Notice that even if the  curvature $\ka(\cdot,0)$ of the initial data does not converge to zero at infinity (i.e. as $\theta \to 0$ or $\theta   \to \pi$), Theorem \ref{thm-curvaturedecay} shows that $\ka^\alpha$ decays in a sufficient H\"older norm at the two boundary points  after some finite time. 

\medskip

The derivative decay estimates in Section 3 are conducted in  Euclidean space by using  an extrinsic cut-off function up to the critical exponent $\alpha >\frac{1}{3}$. Note that the local estimate does not depend on the global structure, asymptotic lines. Hence, the local estimates are naturally obtained up to $\alpha>\frac{1}{3}$. In the critical case $\alpha=\frac{1}{3}$, one  would need to introduce an affine-invariant cut-off function.

\medskip

To apply the derivative estimate with the arclength parameter $s$, we have to use the change of variable $\p_s=\ka \, \p_\theta$. Therefore, we need to derive a lower bound for $\ka$. We do so by considering  the flow as a fast diffusion equation.  Then, for $\alpha>\frac{1}{2}$ we obtain the required lower bound in  Theorem \ref{thm-decay}.

\medskip

In the last section, we show (by utilizing our  estimates in previous sections)  that $J(t)$ converges to zero as time tends to infinity on each compact interval in $(0,\pi)$. Thus, $\ka^\alpha_{\theta\theta}-\ka^\alpha$ converges to zero in $L^2$-sense  (see in Lemma \ref{lemma-main}). We then  conclude the convergence of  $\ka^\alpha(\theta,t)$  to $c\, \sin\theta$ in the  $C^{\infty}_{\text{loc}}$-topology,  for some $c>0$  depending  on  the width of the smallest slab region which encloses our solution.  
This   yields  Theorem \ref{thm-goal}. Finally,  Theorem \ref{thm-goal} combined with  Proposition \ref{prop-basic} 
 implies our main result Theorem \ref{thm-main}.

\section{Preliminaries and Curvature Estimates}\label{sec-pre}

We begin by defining the following  notation. We denote by $\mathcal{N}_t$ the {\em normal image} of $M_t$
at a given instant $t$, namely: 
\be
\mathcal{N}_t :=\{ n\in S^1 : n\; \text{is an inward unit normal vector to} \; M_t\}
\ee
and  denote by $S(n,t): \mathcal{N}_t \to \R$ the {\em support function } 
\be\label{eqn-support}
S(n,t) :=\sup_{X\in M_t} \langle -n, X\rangle.
\ee

In the next Proposition we gather some  basic properties of any solution $M_t$ to the $\alpha$-CSF
which satisfies the assumptions of Theorem \ref{thm-main} and sketch its proof for the reader's convenience. 

\begin{proposition}\label{prop-basic}
Assume that $M_0$ is a strictly convex smooth non-compact complete curve embedded in $\mathbb{R}^2$
such that its  two ends are asymptotic to the two lines $\{x_1=\pm 1\}$,  as $ x_2\to +\infty$. Then, the $\alpha$-CSF {\em($\alpha>0$)} has a unique convex complete solution $M_t$ existing for all time $t \in [0,+\infty)$. Moreover, each $M_t$ is a graph over $(-1,1)$ with  $\mathcal{N}_t=\{ \langle n,e_2\rangle >0\}$.
\end{proposition}

\begin{proof}
First, by the strict convexity and the completeness of $M_0$, $N_0$ is open in $S^1$ and we easily see that  $M_0$ is a convex graph over $(-1,1)$.
\smallskip

Next, we claim that a complete convex solution $M_t$ (if it exists) remains as a graph for a short time $t\in [0,T]$.   We consider closed circular solutions \[\Gamma_t^h=\{x \in \R^2: (1/2)^{\alpha+1}=|x-(0,h)|^{\alpha+1}+(\alpha+1) t\}\mbox{ for }h \in \R.\] Since the convex hull of $M_0$ contains $\Gamma^h_0$ for $h \gg 1$, the convex hull of $M_t$ contains  $\Gamma^h_t$ for $h \gg 1$. Let $T$ be the singular time of $\Gamma^h_t$. Then, the convex hull of $M_t$ contains $\Gamma^h_t$ for $t\in [0,T]$ for $h \gg 1$. This implies $N_t \cap \{ \la n ,e_2\ra < 0\} = \phi$ for $t\in[0,T]$. $M_t$ is strictly convex by the strong maximum principle and hence again $N_t$ is open in $S^1$. Therefore, $N_t \cap  \{ \la n ,e_2\ra \le 0\} = \phi$. i.e. it is a graph. 
\smallskip 

The all  time existence of complete convex graph solutions is  given in \cite{CDKL}. Moreover, it was also shown in \cite{CDKL} that the domain of every graphical solution is fixed over time. Therefore, each $M_t$ is a convex complete graph over $(-1,1)$. Since $M_t$ is a complete convex graph over $(-1,1)$, it follows that $N_t=\{\la n , e_2 \ra >0\}$. 
\smallskip

Finally, let us sketch the proof of the  uniqueness assertion of the proposition. Let $M_t$ and $\bar M_t$ be  two solutions with the same initial data $M_0=\bar M_0$. 
We may assume that the convex hull of $M_0$ contains the origin.  Consider, for $\epsilon \in (0,1)$ the rescaled solution $\hat{M}_t:=(1-\epsilon)\bar M_{(1-\epsilon)^{-(1+\alpha)}t}$.  Then, each $\hat{M}_t$ is a graph over $(-1+\epsilon ,1-\epsilon)$  and the convex hull of $M_0$ contains  $\hat{M}_0$. Thus, the convex hull of $M_t$ contains $\hat{M}_t$ by the comparison principle. Passing $\epsilon \downarrow 0$, we conclude that the convex hull of $M_t$ contains $\bar M_t$. Similarly, the convex hull of $\bar M_t$ contains $M_t$, yielding that the  solution is unique.
\end{proof}

\begin{lemma}\label{lem-support}
Assume that  $M$ is a  complete graph of a smooth strictly convex function defined on  $(-1,1)$ which implies that for any unit vector $n$ satisfying $\langle n,e_2 \rangle >0$,  there exists a point $X(n)   \in M$ such that $n$ is the inner unit normal at $X(n)\in M$. Then, we have $S(n) :=\sup_{X\in M }\langle -n,X \rangle=\langle -n,X(n)\rangle$ and
 \begin{align*}
 \lim_{n\to \pm e_1} S(n)=1.
 \end{align*}
\begin{proof}
We will  only show that  $\lim_{n\to e_1} S(n)=1$, as the other limit follows similarly. Let  $X(n):=(x_1(n),x_2(n))  $.
If $x_1(n)$  is sufficiently close  to $-1$, we have $x_2(n)> 0$, $0<\langle n,e_1 \rangle <1$. Thus, since   $\langle n,e_2 \rangle >0$ we have  
 $$\limsup_{n\to e_1}S(n)=\limsup_{n\to e_1}\langle  -n,X(n) \rangle =\limsup_{n\to e_1}-x_1(n) \, \langle n,e_1 \rangle -x_2(n) \, \langle n,e_2 \rangle \leq 1.  $$
Now, we assume that there exists a sequence of unit vectors $n_i$  such that $\langle n_i,e_2\rangle > 0$, $\lim_{i\to\infty} n_i=e_1$, and $S(n_i) \leq 1-\epsilon$ for some $\epsilon>0$. We denote by $L_i$ the tangent line to $M$ at $X(n_i)$. We observe that there exists the closed half plane $E_i \subset \mathbb{R}^2$ such that $\p E_i \parallel L_i$, $L_i \subset E_i$, and $-(1-\epsilon) \, n_i \in \p E_i$. Then, we have $M \subset E_i$ and $\overline{\lim_{i\to\infty} E_i} =\{x_1 \geq -1+\epsilon \}$. To be more precise, for every $X'=(x'_1,x'_2)\in \R^2$ with $x'_1<-1+\e$,  $X' \notin E_i $ for large $i$. This contradicts  the condition that $M$ is a graph over $(-1,1)$.
\end{proof}
\end{lemma}

After  scaling and rotating our initial  data $M_0$,  Proposition \ref{prop-basic} implies that we only need to prove the following result instead of Theorem \ref{thm-main}. 

\begin{theorem}[Local convergence to solitons]\label{thm-goal}
Let $M_0$ be  a strictly convex smooth non-compact complete curve embedded in $\mathbb{R}^2$
such that its  two ends are asymptotic to the two  lines $\{x_1=\pm 1\}$,  as $ x_2\to +\infty$.  For  any
$\alpha >1/2$,  let 
$M_t$, $t \in (0,+\infty)$  be  the unique solution of the $\alpha$-CSF  with the initial data $M_0$
and denote by  $f:(-1,1) \times [0,+\infty)\to \R$  the graphical parametrization of $M_t$. 

Then, the gradient $f_x(x,t)$ converges to $f_\alpha'(x)$ in $C^\infty_{\text{loc}}[(-1,1)]$ as $t\to +\infty$, where the graph of the function $ f_\alpha(x)=\int_0^x  f_\alpha'(s)\, ds $ is the translating soliton to the $\alpha$-CSF moving in $e_2$ direction whose two ends are asymptotic to $\{x_1=\pm1\}$.
\end{theorem}

\subsection{Parametrization of a convex curve by its normal vector}

Let $M\subset \R^2$ be a strictly convex $C^2$ curve which is the  boundary of a convex body $\hat M \subset \mathbb{R}^2$.  We denote by $n$  the normal vector at $X=(x_1,x_2)\in M$ and $\theta \in [0,2\pi)$ the angle between $n$ and $e_1$.  This  parametrization was used in Gage-Hamilton \cite{GH}. Note that a  {\em convex}  curve is completely determined by the curvature function parametrized by $\theta$, namely $\kappa(\theta)$, up to a translation. 

Recall the well known facts that the  arc-length parameter $s$ satisfies  ${\displaystyle \ka = \fr {\p \theta}{\p s}}$, thus
\[\fr{\p X}{\p \theta} = \fr{\p X}{\p s}\fr{\p s}{\p \theta}= \fr{1}{\kappa}\fr{\p X}{\p s}   \]  and \[\fr{\p X}{\p s}  = (\cos(\theta-\fr{\pi}{2}) , \sin (\theta - \fr{\pi}{2} )=(\sin\theta, -\cos\theta)\] yielding  \be \label{eq-curveintheta}X(\theta_1)-X(\theta_0) = \left(\int_{\theta_0}^{\theta_1} \fr{\sin\theta}{\ka(\theta)} d\theta , \int_{\theta_0}^{\theta_1}- \fr{\cos\theta}{\ka(\theta)} d\theta \right).\ee

As mentioned earlier, J. Urbas \cite{U2} showed that  for exponents $\alpha>1/2$, all the translators of the $\alpha$-GCF (which includes  the $n=1$ case of  the $\alpha$-CSF) are enclosed inside  a cylinder.  Moreover, $M_0$ is a translating soliton 
of the $\alpha$-CSF moving in $e_2$ direction with the speed $c>0$  if and only if  $\ka^\alpha = \la \, n, c \, e_2 \ra = c\, \sin\theta$. Let us observe next that this fact  and \eqref{eq-curveintheta} give a short proof of Urbas's result when $n=1$. 

\begin{proposition}\label{prop-translator}
For $\alpha >1/2$, there exists a strictly convex function $f_\alpha:(-1,1)\to \mathbb{R}$ such that $\displaystyle \lim_{|x|\to 1} |f_\alpha'(x)|=+\infty$ and the graph of $f_\alpha$ is a translating soliton to the $\alpha$-CSF.  $f_\alpha$ is unique up to  addition by  a constant. Moreover,  
$$\lim_{|x|\to 1} |f_\alpha (x)|=+\infty, \quad \mbox{if}\,\, \alpha \in (\frac{1}{2},1]
\qquad \mbox{and} \qquad  \lim_{|x|\to 1} |f_\alpha (x)|=C < +\infty, \quad \mbox{if}\,\, \alpha > 1.$$ For $\alpha \in (0, 1/2]$, translating solitons are entire graphs on $\R$. 
\end{proposition}

\begin{proof}
Given $\alpha>\frac{1}{2}$, we define the  positive finite constant $m(\alpha)$ by
\begin{align}\label{eq-speed}
m(\alpha):=\Big (\int^{\pi}_0 \fr{\sin y}{\sin ^{1/\alpha} y} dy\Big )^\alpha.
\end{align}
If we fix a point $X_\alpha(\fr{\pi}{2})=(x^1_\alpha,x^2_{\alpha})(\fr \pi 2)= (0,0)$, the equation $\ka^\alpha(\theta) = m(\alpha)\sin\theta$ defines a translating soliton of the $\alpha$-CSF by \eqref{eq-curveintheta}. Namely, $x^i_\alpha :(0,\pi)\to \mathbb{R}$ for $i=1,2$ by 
\bee
 x^1_\alpha(\theta)=m(\alpha)^{-\frac{1}{\alpha}}\int^{\theta}_{\pi/2}(\sin y)^{1-\frac{1}{\alpha}} dy,\qquad  x^2_\alpha(\theta)=-m(\alpha)^{-\frac{1}{\alpha}}\int^{\theta}_{\pi/2} (\sin y)^{-\frac{1}{\alpha}}  \cos y \, dy.
\eee
Note that we have $x^2_\alpha \geq 0$, $x^1_\alpha \in (-1,1)$, $\displaystyle\lim_{\theta\to 0}x^1_\alpha(\theta)=-1$, $\displaystyle \lim_{\theta\to \pi}x^1_\alpha(\theta)=1$. The graph of $(x^1_\alpha, x^2_\alpha)$ could be written as a graph of a function $f_\alpha$ on $(-1,1)$. All the other properties of $f_\alpha$ can be checked directly from $x^i_\alpha$. Note the the speed $m(\alpha)$ is fixed, as we have fixed the size of the interval $I:=(-1,1)$ over which our translator $f_\alpha$ is defined. For $\alpha\in(0,1/2]$, $m(\alpha)=\infty$ implies every soliton has to be an entire graph. 
\end{proof}

\subsection{Evolution equations}

We first recall  well-known equations for the  normal vector $n(p,t)$, the speed
$\bar \kappa^\alpha(p,t)$ and the extrinsic distance   $|X(p,t)|$, where all are considered with respect to the 
geometric parametrization which defines the flow in \eqref{eqn-CSFG}, in particular $\p_s$ and $\p_{ss}$ denote as usual 
the first and second order derivatives with respect to arc-length parameter $s$. The base point of this arc-length could be any point, but we choose an orientation of this parameter $s$ in such a way that $\fr{\p\theta}{\p s} = \ka$. 

\smallskip

\noindent {\em Evolution  of the normal:} 
\be{\displaystyle  \p_t n = -\nabla \bar\ka^\alpha = -\alpha  \, \ka^{\alpha-1}\bar\ka_s}\,\p_s \quad\text{ or equivalently}\quad\p_t\theta=\alpha  \, \ka^{\alpha-1}\bar\ka_s .\ee

\noindent {\em Evolution  of the speed $\bar \kappa^\alpha$:}  \be\label{eqn-speed} {\displaystyle (\p_t - \alpha \bar \ka^{\alpha-1} \p_{ss} )  \bar\ka^\alpha =  \alpha  \,\bar  \ka^{2\alpha+1} }.\ee

\noindent {\em Evolution  of the curvature  $\bar \kappa$:} 
\be\label{eqn-cur}
\bar \ka_t =  \p_{ss}  \bar\ka^{\alpha}  + \bar\ka^{\alpha+2}.
\ee

\noindent {\em Evolution  of the extrinsic distance:} 
 \be\label{eqn-X2}  (\p_t - \alpha \bar \ka^{\alpha-1} \p_{ss} )|X|^2 = 2\, \alpha \bar\ka^{\alpha-1}(-1+(\alpha^{-1}-1) \la X,n \ra \, \bar\ka). \ee

\smallskip
Next, we will  compute the evolution of the derivatives of the speed  $\kappa^\alpha$ 
by differentiating equation \eqref{eqn-speed}.  Before this, let us note  that the parameter $s$ is not a fixed coordinate and changes with respect to time. In fact, 
$$\fr{\p}{\p s} =\fr{1}{\sqrt{g_{11}}}\fr{\p}{\p x_1}, \qquad   \fr{\p^2}{\p t\p s} =\fr{\p^2}{\p s\p t}- \fr{\p_t \, g_{11}}{2 g_{11}\sqrt{g_{11}}}
 \, \fr{\p}{\p x_1}$$
and hence   the commutator satisfies 
\be \label{eqn-comm} \fr{\p^2}{\p t\p s}=\fr{\p^2}{\p s\p t}+\bar\ka^{\alpha+1} \fr{\p}{\p s}.
\ee
To simplify the notation we set   $u:=\bar\ka^\alpha$, and express  equation   equation \eqref{eqn-speed}
as 
\begin{align}\label{eq-u} 
 u_t = \alpha \, u^{1-\frac{1}{\alpha}}u_{ss}+ \alpha u^{2+\frac{1}{\alpha}}=\alpha u^{1-\fr{1}{\alpha}}\,  \big (u_{ss}+u^{1+\fr{2}{\alpha}} \big )
\end{align}
Differentiating \eqref{eq-u} while using the  commutator identity \eqref{eqn-comm}, we obtain the following evolution equations for  the higher order derivatives of $u$:
\begin{align}\label{eq-us}
\p_t u_s =\p_s u_{t}+u^{1+\fr{1}{\alpha}}u_s =\alpha u^{1-\frac{1}{\alpha}}\p^2_{ss}u_s+(\alpha-1)u^{-\frac{1}{\alpha}}u_s u_{ss}+ 2 (\alpha+1)\, u_su^{1+\frac{1}{\alpha}} 
\end{align} and 
\be\ba\label{eq-uss} \p_t u_{ss}&=\p_t\p_s u_s=\p_s\p_t u_{s} +u^{1+\fr{1}{\alpha}}u_{ss}\\
&=\alpha u^{1-\fr{1}{\alpha}}(u_{ss})_{ss} +2(\alpha-1) u^{-\fr{1}{\alpha}}u_su_{sss} + (\alpha-1)u^{-\fr{1}{\alpha}} u_{ss}^2+(\fr{1}{\alpha}-1)u^{-1-\fr{1}{\alpha}} u_s^2 u_{ss}\\
&\quad \quad+2(\alpha+1)(1+\fr{1}{\alpha}) u^{\fr{1}{\alpha}}u_s^2+(2\alpha+3) u^{1+\fr{1}{\alpha}}u_{ss}. \ea\ee

For a smooth strictly convex solution, $\theta(p,t)$ is a smooth invertible function. Thus,  for a fixed $\theta'$ in the image of $\theta(p,t)$ for a time interval $t\in I$, we may define a curve $\gamma_{\theta'} (t)$ for $t\in I$ so that $\theta(\gamma_{\theta'}(t), t) =\theta'$.
Let us parametrize the curvature $\bar \ka$ by $(\theta,t)$ as follows

\[  \ka (\theta,t)=\bar\ka(\gamma_\theta(t),t). \] 

We will often abuse the  notation and continue to use $\ka(n,t)=\ka(\theta,t)$,  for $n=(\cos\theta,\sin\theta)$. Let us next  derive the  evolution equation of $\ka(\theta,t)$.  Note that   \[\p_t \ka = \p_t \bar\ka + \p_s \bar\ka \, \dot \gamma_\theta,  \qquad\text{where }\dot \gamma_\theta=\fr{\p}{\p  t} (s( \gamma_\theta(t))).\]
On the other hand,  since $\theta(\gamma_{\theta'}(t),t)$ is   constant in $t$ we have  \[0=\fr{d}{dt} \theta (\gamma_{\theta'} (t),t) =\bar\ka(\gamma_{\theta'}(t),t) \, \dot \gamma_{\theta'} +\p_s \bar\ka^{\alpha}  (\gamma_{\theta'}(t),t) \quad \mbox{thus} \quad \dot \gamma_\theta = -\alpha \bar\ka_s\bar\ka^{\alpha-2}.\] Hence   \be \label{eq-kappas}\p_t \ka= \p_t\bar\ka -\alpha \bar\ka^{\alpha-2}\bar\ka_s^2\ee 
and use $\p_s =  \ka \, \p_\theta$ to conclude that 
\[\ba\p_t  \ka &= ( \ka^\alpha)_{ss} + \ka^{\alpha+2} -\alpha  \ka^{\alpha-2} \ka_s^2 \\
&= \alpha \ka^{\alpha+1}   \ka_{\theta\theta} + \alpha(\alpha-1)\ka^{\alpha}\ka_\theta^2+\ka^{\alpha+2}\\
&=\ka^2 \big ((\ka^{\alpha})_{\theta\theta} + \ka^\alpha \big )\ea \]
which also implies the equation 
\begin{align}\label{eq-k^a}
\p_t (\ka^\alpha)=\alpha(\ka^\alpha)^{1+\frac{1}{\alpha}} ((\ka^{\alpha})_{\theta\theta} + \ka^\alpha).
\end{align}
The derivation of equation \eqref{eq-k^a} is well known, however we included  it here for the reader's convenience. Sometimes, it is useful to define $p:= \ka^{\alpha+1}$ which we call the  {\em pressure function} following  the terminology of the porous medium and fast-diffusion equations.  The  evolution of $p(\theta,t)$ is given by  \be\label{eq-pressure}\p_t p = \alpha \, p p_{\theta\theta} - \fr{\alpha}{\alpha+1}p^2_\theta +(\alpha+1) \, p^2 .\ee

\subsection{Harnack Estimates}

We need a following pointwise Harnack estimate in $(\theta,t)$ variables  derived from Li-Yau-Hamilton differential Harnack estimate which appears in \cite{Ha} and \cite{C} for the mean curvature flow and the $\alpha$-Gauss curvature flow, respectively. 
\begin{prop}[Harnack Estimate] \label{prop-harnack} Let $M_t$ be a smoothly strictly convex solution of the $\alpha$-CSF.  Then, the curvature $\kappa(\theta,t)$ satisfies   \[\ka_t\ge   \fr{1}{\alpha+1} \fr{\ka}{t}\] 
implying for $0<t_1<t_2$ the inequality  \[\ka(\theta,t_2) \ge \Big (\fr{t_1}{t_2}\Big )^{\fr{1}{\alpha+1}} \ka(\theta,t_1).\]
\begin{proof}From 478 page in \cite{C}, for $\bar\ka (p,t)$,  \[\alpha \fr{\bar\ka_t}{\bar\ka} - \alpha^2\fr{\bar\ka_s^2}{\bar\ka^2}\fr{\bar\ka^\alpha}{\bar \ka} \ge -\fr{1}{1+\alpha^{-1}}\fr{1}{t}.\] Since \[\p_t \bar \ka= \p_t\bar\ka -\alpha \bar\ka^{\alpha-2}\bar\ka_s^2\] this directly implies the proposition.
\end{proof}
\end{prop}

%\begin{remark}
%Actually it seems like this could be generalized into higher dimensional GCF and we get \[\ka(\theta,t_2) \ge \Big (\fr{t_1}{t_2}\Big )^{\fr{n}{n\alpha+1}} \ka(\theta,t_1)\quad\mbox{for }\theta\in S^n.\]
%\end{remark}

\subsection{Curvature Upper and Lower bounds}
The goal of this section is to prove Proposition \ref{prop-curvbdd}, which gives global upper bounds on the speed 
$\kappa^\alpha$ for $t >0$ and local (in $\theta$) lower bounds on the speed $\kappa^\alpha$ for  large times. 
We first show a simple lemma which says that the support functions of convex surfaces are ordered if one surface contains the other. 
\begin{lemma}\label{lem-geom}
Suppose $M_1$ and $M_2$ are convex hypersurfaces in $\R^{n+1}$ and the convex hull of $M_1$ contains $M_2$. Then, their support functions $S_i(n)=\sup_{x \in M_i} \langle -n,x\rangle$ satisfy $S_1(n) \geq S_2(n)$. 
\begin{proof}
We denote by $E_i \subset \mathbb{R}^3$ the convex hull of $M_i$. Then, we have $\displaystyle S_0=\sup_{x \in M_i} \langle -n,x\rangle=\sup_{x \in E_i} \langle -n,x\rangle$  by the convexity. Hence, $E_2 \subset E_1$ implies the desired result.
\end{proof}
\end{lemma}

%\begin{remark}Though we don't need this fact in the paper, we actually proved more general lemma which applies to non smooth convex surfaces. For any two supporting hyperplanes $l_1$ and $l_2$ of $M_1$ and $M_2$, respectively, with the same inward normal vector $n_0$, $l_2$ is more located toward $n_0$-direction than $l_1$.  
%
%\end{remark}

\begin{proposition}\label{prop-curvbdd}
Let $M_t$ be a solution of the $\alpha$-CSF as in Proposition  \ref{prop-basic}. Then, given  $t_0>0$, we have
\begin{align*}
 \lim_{ n \to \pm e_1}\sup_{s\in [t,t+3]} \ka^\alpha(n,s)=0 \qquad \text{ and} \qquad \ka^\alpha(n,t)\le C
\end{align*}
for all \ $n \in \mathcal{N}=\{\langle n,e_2\rangle>0\}$ \ and \ $t\ge t_0$, \ where the constant  $C$ depends on $t_0$ and $M_0$. \\ In addition, for each $\delta \in (0,\frac{1}{10})$, there is a large $T > 0$ and $c(\delta)>0$ such that  $$  \ka^\alpha(n,t) \ge c(\delta) $$ whenever $\langle n,e_2 \rangle \geq \delta$ and $t \ge T$. The constants $T$ and $c(\delta)$ may depend on $M_0$ and $\delta$.

\begin{proof} 
We begin by observing that the support function $S(n,t)$ of a  solution $M_t$ of  the $\alpha$-CSF (defined by
\eqref{eqn-support})  satisfies  $\partial_t S(n,t) = -\ka^\alpha (n, t)$. 

Therefore, the  Harnack inequality \ref{prop-harnack} and the above observation yield
\begin{align*}
c(t_0) \, \ka^\alpha(n,\tau) \leq   \int_{\tau}^{\tau+1}\ka^\alpha(n,s)\, ds =S(n,\tau)-S(n,\tau+1)\leq S(n,t)-S(n,t+4),
\end{align*} 
for $\tau \in [t,t+3] $ and $t \geq t_0>0$. Since $\displaystyle\lim_{n \to \pm e_1} S(n,t)=\lim_{n \to \pm e_1} S(n,t+4)= 1$ by Lemma \ref{lem-support}, we have the first desired result
\begin{align*}
\lim_{n\to\pm 1}\sup_{\tau\in [t,t+3]}\ka^\alpha(n,\tau)=0.
\end{align*}

\bigskip

Given $\alpha \in (\frac{1}{2},1)$, we denote by $M_\alpha$ the translator $M_\alpha=\{(x,f_\alpha(x)): |x|<1\}$, where $f_\alpha$ is given in Proposition \ref{prop-translator}. For $\alpha>0$, we define $M_\alpha$ by 
\begin{align*}
M_\alpha=\{(x_1,f_\alpha(x_1)): |x_1|<1\}\cup \{(\pm 1, x_2):x_2 \geq \lim_{|y|\to 1}  f_\alpha(y)\}.
\end{align*} 
Let us fix a small $\e_0\in(0,1/10)$. Then depending on $M_0$, there is $L>0$ so that the convex hull of $${\hat M}_\alpha:=\fr{1}{(1-\e_0)^{1/\alpha}} \, M_\alpha - L e_2 $$ contains initial surface $M_0$ and $$\bar{M}_\alpha:=\fr{1}{(1+\e_0)^{1/\alpha}}\, M_\alpha+L  e_2$$ is contained in the convex hull of $M_0$. Then,  $\hat M_t^\alpha \coloneqq {\hat  M_\alpha}+(1-\e_0) \, m t\, e_2$  and $\bar M_t^{\alpha}\coloneqq \bar{M}^\alpha + (1+\e_0) \, m t\, e_2$ are solutions of the $\alpha$-CSF, where $m=m(\alpha)$ is the positive constant given in \eqref{eq-speed}.
 
Let us denote ${\hat S}_{\alpha}$ and ${\bar S}_{\alpha}$ by the support functions of the outer barrier $\hat M_t^\alpha$ and the inner barrier $\bar M_t^{\alpha}$, respectively. Thus $\p_t {\bar S}^\alpha = -(1+\e_0) m \, \langle n,e_2\rangle$,  and $\p_t {\hat S}^\alpha = -(1-\e_0)m \, \langle n,e_2\rangle$. Moreover, if  $K=\sup\{ S(n,0):\langle n,e_2\rangle>0\}$,  we have 
\be\label{eqn-K1}
0<{\hat S}^\alpha(n,t) - {\bar S}^\alpha(n,t) \le 2(L+\e_0 \, mt)\langle n,e_2\rangle+ \Big ( \fr{1}{(1-\e_0)^{1/\alpha}} -\fr{1}{(1+\e_0)^{1/\alpha}}\Big)K
\ee
for all $(n,t)\in \mathcal{N}\times (0,\infty)$.  Let us set  $M:=  \Big ( \fr{1}{(1-\e_0)^{1/\alpha}} -\fr{1}{(1+\e_0)^{1/\alpha}}\Big ) $ from now on. 
 
Next, by the comparison principle and Lemma \ref{lem-geom}, ${\bar S}^\alpha \le S \le {\hat S}^\alpha$ for all $(n,t)\in \mathcal{N}\times (0,\infty)$. Now, if $\ka^\alpha(n_0,t_0) = C$ then by the Harnack estimate, $\ka^\alpha(n_0,t)\ge \eta \, C$ for  $t\in[t_0,4t_0]$  and  some $\eta=\eta(\alpha)\in(0,1)$. By \eqref{eqn-K1},  \be 0\le  S-{\bar S}^\alpha \le {\hat S}^\alpha-{\bar S}^\alpha\le 2 (L+\e_0 \, m t)+MK \ee on $(n,t)\in \mathcal{N}\times (0,\infty)$   and hence using that $\p_t  S = - \kappa^\alpha$, we obtain 
\be\ba0\le S(n_0,4t_0)-{\bar S}^\alpha(n_0,4t_0) &= S(n_0,t_0)-{\bar S}^\alpha(n_0,t_0)+ \int_{t_0}^{4t_0} \p_t(S - {\bar S}^\alpha)(n_0,t) dt \\
 & \le   2(L+\e_0\, m t_0)+M K +  \big ( -\eta {C} +(1+\e_0) \, m \big )\, 3t_0.\ea\ee
Now we observe  that there is a constant $C(t_0)$ such that $C\ge C_0(t_0)$ makes last line negative. i.e. contradiction. It is also clear that such a $C(t_0)$ can be made uniformly bounded as $t_0\to \infty$. This proves the uniform curvature upper bound 
\begin{align*}
\ka^\alpha(n,t) \leq C(t_0,M_0).
\end{align*} 
 
\bigskip

We suppose next that $\ka^ \alpha(n_0, t_0)=c$,  for some $\langle n_0,e_2\rangle \geq \delta$ with $\delta \in (0,\frac{1}{10})$. By the Harnack estimate, $\ka^\alpha(n_0,t)\le c \, \delta^{-\frac{\alpha}{1+\alpha}}$  for $t\in[ \delta t_0,t_0]$. Similar computation yields 
\begin{align*}
0\le {\hat S}^\alpha(n_0,t_0)-S(n_0,t_0) &= {\hat S}^\alpha(n_0,\delta t_0)-S(n_0,\delta t_0)+ \int_{\delta t_0}^{t_0} \p_t( {\hat S}^\alpha-S)(n_0,t) dt \\
 & \le   2(L+\e_0m \delta t_0)+MK +  \Big ( c\, \delta^{-\frac{\alpha}{1+\alpha}} t_0-(1-\e_0)m \, \la n_0, e_2\ra (1-\delta)t_0 \Big )\\
 \text{(since $0<\e_0,\, \delta<1/10$)} \quad &\le 2L+\fr{1}{5}\delta mt_0+ MK + \Big (c\, \delta^{-\frac{\alpha}{1+\alpha}} - \fr{81}{100}m\delta\Big ) t_0\\
 & \le 2L +MK+ \Big (c\, \delta^{-\frac{\alpha}{1+\alpha}} -\fr{1}{2}m\delta\Big ) t_0.
\end{align*}
Now, we set $T=4(m\delta)^{-1}(2L+MK+1)$. Then, for $t_0 \geq T$ we must have
\begin{align*}
c\, \delta^{-\frac{\alpha}{1+\alpha}} \geq \frac{1}{4}m\, \delta
\end{align*}
in order to satisfy the  inequality above. We conclude that 
$$\kappa^\alpha(n_0,t_0) =c  \geq \frac 14 \, m \, \delta^{1+\frac{\alpha}{1+\alpha}}:=c(\delta)$$
completing the proof of the last claim of our proposition. 
\end{proof}
\end{proposition}

\subsection{Barrier Construction} 

Based on our uniform curvature bound given in Proposition \ref{prop-curvbdd} and the fact that $\lim_{\theta\to 0} \kappa =0$, the following barrier shows that the  modulus of continuity of $k(\theta,t)$ at $\theta=0$ is 
$\kappa(\theta,t)= O(\kappa^{1-\e})$,  for every $\e>0$ and   $t \geq t_1 \gg 1$ .
\begin{lemma}\label{lem-barrier}For every  $t_0>0$ with $t_0< \min(3, \fr{6}{1+\fr{1}{\alpha}})$, there is $A_0>0$ such that for every $A> A_0$, the function defined by \begin{equation*}
h_\delta:=A\sin^{t/3} \Big (\fr{3\theta}{t}\Big ) +\delta \quad \text{ on }\theta\in \Big (0, \fr{t\pi}{6}\Big ), \,\,  t\in(0,t_0]
\end{equation*}
%Then, \begin{align*}v(\theta,t)&=\min (u(\theta,t)+\delta, c(t):=\Big ({23}/{5}-2t\Big )^{-1/2} , u(\pi-\theta,t)+\delta)\\
%&=\begin{cases}(\sin (3 \theta t^{-1}) )^{t/3}+\delta & \quad \text{if}\; \theta \in (0,\frac{\arcsin(c(t)-\delta)}{3}t] \\
%\Big (\fr{23}{5}-2t\Big )^{-1/2} & \quad \text{if}\; \theta \in (\frac{\arcsin(c(t)-\delta)}{3}t,\pi-\frac{\arcsin(c(t)-\delta)}{3}t)\\ 
%(\sin (3 (\pi-\theta) t^{-1}) )^{t/3}+\delta & \quad \text{if}\; \theta \in (\pi-\frac{\arcsin(c(t)-\delta)}{3}t,\pi) \end{cases}
%\end{align*}
is a viscosity supersolution of \eqref{eq-k^a} for all $0<\delta \le\delta_0(t_0,A)$. \end{lemma}

\begin{proof} Set 
$$w(\theta,t)=A\, \varphi(\theta,t)^{t/3} \qquad \mbox{where} \quad \varphi(\theta,t):= \sin (\fr{3\theta}{t}).$$
Then,  for $\theta \in (0,\frac{\pi}{6}t)$ and $0<t\le t_0$, we compute 
\begin{align*}
w_\theta =A \, \varphi^{t/3-1}  \cos (3 \theta t^{-1}) 
\end{align*}
and 
\begin{align*}
w_{\theta\theta}& =A\, (1-3/t)\, \varphi^{t/3-2} \, \cos^2  (3 \theta t^{-1})  -A\, (3/t)\, \varphi^{t/3} \\
& =A\, (1-3/t)\, \varphi^{t/3-2} -A\, \varphi^{t/3} = A\, (1-3/t)\, \varphi^{t/3-2} - w. 
\end{align*}
Therefore,
\begin{align*}
\alpha \, w^{1+\fr{1}{\alpha}}\, \big (w_{\theta\theta} +w \big )  =\alpha \, A^{2+\fr{1}{\alpha}}\, (1-3/t)\varphi^{\fr{2\alpha+1}{3\alpha}t-2}.
\end{align*}
On the other hand, expressing $w/A=\exp \big ( (t/3) \log \varphi \big  )$, we have 
\begin{align*}
\fr{w_t}{A} & =\frac{1}{3}\varphi^{t/3}\Big ( -(3 \theta t^{-1})\cot (3 \theta t^{-1})+\log \varphi \Big )  \\
& \geq -\frac{1}{3} \varphi^{t/3}+\frac{1}{3} \varphi^{\fr{2\alpha+1}{3\alpha}t-2}\varphi^{2-\fr{1+\alpha}{3\alpha}t}\log \varphi \\& \geq -\frac{1}{3} \varphi^{\fr{2\alpha+1}{3\alpha}t-2}-\frac{1}{3} \varphi^{\fr{2\alpha+1}{3\alpha}t-2} \fr{3\alpha}{(6\alpha-(1+\alpha)t)e}\\
&\geq -\frac{1}{3} \Big ( 1+\fr{3}{(6-(1+\fr{1}{\alpha})t)e} \Big ) \, \varphi^{\fr{2\alpha+1}{3\alpha}t-2}. 
\end{align*}
Combining the  two inequalities, yields 
\begin{align*}
w_t-\alpha \, w^{1+\fr{1}{\alpha}}(w_{\theta\theta} +w)\ge \fr{A}{3} \Big  [\Big (\fr{3}{t}-1\Big )3\alpha A^{1+\frac 1{\alpha}}- 1- \fr{3}{(6-(1+\fr{1}{\alpha})t)e}\Big )\varphi^{\fr{2\alpha+1}{3\alpha}t-2}.\end{align*}
Hence, for ${\displaystyle t_0< \min(3, \fr{6}{1+\fr{1}{\alpha}})}$, there is $\e=\e(t_0,\alpha)>0$ such that \begin{align*}
w_t-\alpha w^{1+\fr{1}{\alpha}}(w_{\theta\theta} +w)\ge \fr{A}{3} \Big  [\e A^{1+\frac 1{\alpha}}- \fr{1}{\e}\Big )\varphi^{\fr{2\alpha+1}{3\alpha}t-2}.
\end{align*}
This proves  there exists  $A_0(t_0,\alpha)>0$ and $\e_0(t_0,\alpha)>0$ such that if $A\geq  A_0$, 
\be\label{eqn-ineqw}
w_t-\alpha w^{1+\fr{1}{\alpha}}(w_{\theta\theta} +w)\ge \e_0\varphi^{\fr{2\alpha+1}{3\alpha}t-2} \ge0
\ee
holds  on $\theta\in(0,\fr{\pi t}{6})$ and $t\in(0,t_0)$. 
\smallskip

For the next step, we set  $h_\delta := w+\delta$  for a  small constant $\delta>0$ and to simplify  the notation we drop
the index $\delta$ from $h$ for the rest of the proof, denoting $h:=h_\delta$. 
Then for $\theta\in(0,\fr{\pi}{6}t)$ and $0<t\le t_0$, we compute 
\begin{align*}
h_t-\alpha h^{1+\fr{1}{\alpha}} (h_{\theta\theta} + h) &= (w_t-\alpha w^{1+\fr{1}{\alpha}} (w_{\theta\theta}+w)) -\alpha (h^{1+\fr{1}{\alpha}}-w^{1+\fr{1}{\alpha}}) \, w_{\theta\theta} -\alpha(h^{2+\fr{1}{\alpha}}-w^{2+\fr{1}{\alpha}}).
\end{align*}
Observe  that, by  Taylor's Theorem, we have 
\begin{align*}
 -\alpha (h^{1+\fr{1}{\alpha}}-w^{1+\fr{1}{\alpha}})\, w_{\theta\theta} -\alpha(h^{2+\fr{1}{\alpha}}-w^{2+\fr{1}{\alpha}})&= -\alpha \Big ( \delta (1+\fr{1}{\alpha}) \bar w^{\fr{1}{\alpha}} w_{\theta\theta}+\delta (2+\fr{1}{\alpha}) \hat w^{1+\fr{1}{\alpha}} \Big ) \\
 &\ge -\delta(\alpha+1) w^{\fr{1}{\alpha}} w_{\theta\theta}-\delta(2\alpha+1) (w+\delta)^{1+\fr{1}{\alpha}}
\end{align*}where we used that $w \leq \bar w, \hat w \leq w+\delta$ and $w_{\theta\theta} \le 0$, $w\ge0$. Hence,
using \eqref{eqn-ineqw} we obtain 
\begin{align*}
h_t-\alpha h^{1+\fr{1}{\alpha}}  (h_{\theta\theta}+h ) &= \big (w_t-\alpha w^{1+\fr{1}{\alpha}} (w_{\theta\theta}+w) \big ) -\alpha (h^{1+\fr{1}{\alpha}}-w^{1+\fr{1}{\alpha}})w_{\theta\theta} -\alpha(h^{2+\fr{1}{\alpha}}-w^{2+\fr{1}{\alpha}})\\
&\ge \e_0\varphi^{\fr{2\alpha+1}{3\alpha}t-2} -\delta(\alpha+1) w^{\fr{1}{\alpha}} w_{\theta\theta}-\delta(2\alpha+1) (w+\delta)^{1+\fr{1}{\alpha}}\\
&\ge \big (  \e_0\varphi^{\fr{2\alpha+1}{3\alpha}t-2}-\delta(2\alpha+1) w^{1+\fr{1}{\alpha}} \big ) \\
&\quad\quad\quad\quad\quad +\delta \big ( -(\alpha+1)w^{\fr{1}{\alpha}}w_{\theta\theta}+ (2\alpha+1) (w^{1+\fr{1}{\alpha}}-(w+\delta)^{1+\fr{1}{\alpha}})\big ) \\
&\ge \big ( \e_0\varphi^{\fr{2\alpha+1}{3\alpha}t-2}-\delta(2\alpha+1) A^{1+\fr{1}{\alpha}}\varphi^{\fr{\alpha+1}{3\alpha}t} \big ) \\
&\quad\quad\quad\quad\quad +\delta \, \big ( -(\alpha+1)w^{\fr{1}{\alpha}}w_{\theta\theta}-(2\alpha+1)(1+\fr{1}{\alpha})\delta (w+\delta)^{\fr{1}{\alpha}} \big ) \\
&\ge \big (  \e_0-\delta(2\alpha+1) A^{1+\fr{1}{\alpha}} \big ) \varphi^{\fr{2\alpha+1}{3\alpha}t-2}\\
&\quad\quad\quad\quad\quad +\delta[-(\alpha+1)w^{\fr{1}{\alpha}}w_{\theta\theta}-(2\alpha+1)(1+\fr{1}{\alpha})\delta (w+\delta)^{\fr{1}{\alpha}} \big ).
\end{align*}
Moreover, using the earlier calculation of $w_{\theta\theta}$ and $t \geq t_0$, we have 
\begin{align*} -w_{\theta\theta}w^{\fr{1}{\alpha}} \ge \Big (A\big (\fr{3}{t_0}-1\big )\varphi^{\fr{t}{3}-2} \Big ) w^\fr{1}{\alpha}\ge A^{1+\fr{1}{\alpha}} \big (\fr{3}{t_0}-1\big )\varphi^{\fr{1+\alpha}{3\alpha}t -2} \ge A^{1+\fr{1}{\alpha}} \big (\fr{3}{t_0}-1\big ). \end{align*} 
The last two inequalities imply that  there is small a $\delta_0=\delta_0(\e_0, A, t_0)= \delta_0(A,t_0)$ such that, for $0<\delta\le\delta_0$, \[h_t-\alpha h^{1+\fr{1}{\alpha}} (h_{\theta\theta}+h) \ge 0\quad\mbox{holds on }\,\, \theta\in(0,\fr{\pi}{6}t),\, \,\, t\in(0,t_0).\]
This completes the proof.  \end{proof}

This   barrier gives the following, important for our purposes,  curvature decay estimate at the two boundary points $\theta=0, \pi$:

\begin{theorem}[Curvature decay]\label{thm-curvaturedecay}For  $\alpha>1/2$  and $t >3$, we have 
\[\ka^\alpha(\theta,t),\,\ka^{\alpha}(\pi-\theta,t) \le C(M_0,\alpha)\, \theta^{\fr{2}{3}}  \qquad \mbox{on} \,\,\,  \theta\in(0,\pi).\]
\begin{proof} Given $\alpha>\frac{1}{3}$, we have $2<\min(3,\frac{6}{1+\alpha^{-1}})$. It suffices to show for any fixed $t_1>1$  the statement holds at $t=t_1+2$. 

Setting $t_0=2$, let $A_0=A_0(t_0)$ be the constant   given in Lemma \ref{lem-barrier}. By Proposition \ref{prop-curvbdd} we can choose a constant $A>\max\{A_0,\sup_{t \geq 1}\sup_{\theta \in (0,\pi)} \ka^\alpha(\theta,t)\}$,
so  that 
$$\tilde u_{t_1}(\theta,t):=\ka^\alpha(\theta,t_1+t) < A+\delta=w_\delta(t\pi/6,t)$$
for $t\in (0,2]$ and $\theta \in (0,\pi)$, where $w_\delta=A\sin^{t/3} \Big (\fr{3\theta}{t}\Big ) +\delta$ as  given in Lemma \ref{lem-barrier}. Moreover, Proposition \ref{prop-curvbdd} implies that there exists a small constant $c(t_1,\delta)$ such that  $$\ka^\alpha(\theta,t_1+t)< \delta \leq  w_\delta(\theta,t)$$ holds for $0<\theta \leq c(t_1,\delta)$ and  $t \in (0,2]$. Let us denote $t_2=\frac{6c}{\pi}$. Then, $t \in (0,t_2]$ satisfies $\theta \in (0,\frac{t\pi}{6}) \subset (0,\frac{t_2\pi}{6})=(0,c)$. Thus,  
\begin{align*}
\ka^\alpha(\theta,t_1+t)< \delta \leq w_\delta(\theta,t_1+t),
\end{align*}
holds for $t\in (0,t_2]$ and $\theta \in (0,\frac{t\pi}{6})$. Hence, by the comparison principle $\ka^\alpha(\theta,t_1+t) \leqq w_\delta(\theta,t_1+t)$ holds for $\theta \in (0,\frac{t\pi}{6})$ and $t\in (0,2]$, which implies $\ka^\alpha(\theta,t_1+2) \leq w_\delta(\theta,t_1+2)$ for  $\theta \in (0,\pi/3)$. Passing $\delta$ to zero in the last inequality, the following holds for $\theta \in (0,\pi/3)$
\begin{align*}
\ka^\alpha(\theta,t_1+2)\leq A \, \sin^{\frac{t_1+2}3} \big (\frac{3\theta}{t_1+2} \big ) \leq A \, \sin^{\frac 23}(\frac{3\theta}2)\leq C\, \theta^{\frac 23}
\end{align*}
where the constant $C$ depends on $M_0$ and $\alpha$. This concludes the proof of our theorem. 
\end{proof}
\end{theorem}

\section{Decay Estimates (Pointwise curvature derivative estimates)}

In this section we will use the curvature decay estimate at the boundary points $\theta=0,\pi$ proven in Theorem
\ref{thm-curvaturedecay} to obtain decay estimates  for the first and the second order derivative at the boundary points $\theta=0,\pi$ for $u:=\bar \kappa^\alpha$. As a consequence we will obtain the estimate in Theorem \ref{thm-decay}
which will allow us to control the boundary terms when we prove our convergence result in the next section. 
We begin with a first order derivative decay estimate. Throughout   this section  we will assume that  $M_t$ is a solution of the $\alpha$-CSF as in Proposition  \ref{prop-basic}. We will only use  the geometric parametrization in terms of arclength, i.e. we will assume that $u=u(s,t)=\bar \kappa^\alpha(s,t)$.  Here and in what follows  $B_r:=\{x\in \R^2 \,|\, |x|<r \}$ denotes  the Euclidean extrinsic ball of radius $r$. 

 \begin{prop}\label{prop-us}   Suppose  $0<u :=\bar \ka ^\alpha\le L $ on $B_1$ for $t \geq 0$. Then for every $\beta>0$ with $\beta< \min ( 1,\alpha^{-1})$, we have  \[|{u_s}| 
\le C(1+t^{-1/2})\,  u^{\beta} \quad \mbox{on }B_{1/2} \] for some $C=C(\alpha, \beta, L)$.

\begin{proof}
Let $\eta$ be a cut-off function with compact support, and denote by $\square \eta$ the term $$\square \eta\coloneqq \p_t \eta - \alpha u^{1-\fr{1}{\alpha}} \p_{ss} \eta$$
which is defined on the support of $\eta$. Consider the  continuous function $w:=t^2\eta^2 u_s^2u^{-2\beta} $ with a fixed constant $\beta>0$ satisfying the condition of the theorem. Then, on the set $\{w>0\}$
\begin{align*}
&\frac{\p_t w}{2w}=\frac{1}{t}+\fr{\p_t\eta}{\eta}+\frac{\p_tu_s}{u_s}-\beta\frac{\p_t u}{u}, &\frac{\p_s w}{2w}=\frac{\p_s\eta}{\eta}+\frac{\p_su_s}{u_s}-\beta\frac{\p_s u}{u}.
\end{align*}
We differentiate the second equation above again 
\begin{align*}
\frac{\p_{ss} w}{2w}-\frac{w_s^2}{2w^2}=\frac{\eta_{ss}}{\eta}-\frac{\eta_s^2}{\eta^2}+\frac{\p_{ss}u_s}{u_s}-\frac{u_{ss}^2}{u_s^2}-\beta\frac{\p_{ss} u}{u}+\beta\frac{u_s^2}{u^2}.
\end{align*}
Combining \eqref{eq-u}, \eqref{eq-us}, and  the equations above yields
\begin{align*}
\frac{\p_t w}{2w} - \alpha u^{1-\frac{1}{\alpha}}\frac{\p_{ss} w}{2w}+\alpha u^{1-\frac{1}{\alpha}}\frac{w_s^2}{2w^2}=&\fr{1}{t} +\frac{\square \eta}{\eta}+\alpha u^{1-\frac{1}{\alpha}} \frac{\eta_s^2}{\eta^2}+\frac{(\alpha-1)u^{-\frac{1}{\alpha}}u_su_{ss} +(2\alpha+2)u_su^{1+\frac{1}{\alpha}}}{u_s} \\
&+\alpha u^{1-\fr{1}{\alpha}}\frac{u_{ss}^2}{u_{s}^2} -\beta\alpha u^{1-\frac{1}{\alpha}} \frac{u^{1+\fr{2}{\alpha}}}{u} -\beta\alpha u^{1-\fr{1}{\alpha}}\frac{u_s^2}{u^2}.
\end{align*}
Given $T >0$, we assume that $w$ attains its nonzero maximum at $(p_0,t_0)$ for $t\in [0,T]$. Then, at the maximum point we have $w(p_0,t_0)>0$, and thus $t_0 > 0$ and the following holds
\begin{samepage}
\begin{align*}
0\leq \fr{1}{t_0} +\frac{\square \eta}{\eta}+(2\alpha+2-\beta\alpha)u^{1+\frac{1}{\alpha}}+\alpha u^{-\frac{1}{\alpha}}I,
\end{align*}
where 
\begin{align*}
I=\Big (1-\frac{1}{\alpha}\Big )\frac{u_{ss}}{u}+\frac{\eta_s^2}{\eta^2}+\frac{u_{ss}^2}{u_s^2}-\beta\frac{u_s^2}{u}.
\end{align*}
\end{samepage}
Moreover, since  $w_s=0$ at the maximum point $(p_0,t_0)$,  the following hold
\begin{align*}
&\frac{u_{ss}}{u} =\frac{u_s}{u}\frac{u_{ss}}{u_s}=\frac{u_s}{u}\Big(\beta\frac{u_s}{u}-\frac{\eta_s}{\eta}\Big)=\beta \frac{u_s^2}{u^2}-\frac{u_s}{u}\frac{\eta_s}{\eta}, && 
\fr{u^2_{ss}}{u_s^2} = \beta^2\fr{u_s^2}{u^2}+ \fr{\eta_s^2}{\eta^2} - 2\beta \fr{u_s}{u}\fr{\eta_s}{\eta}.
\end{align*}
Substituting these derivatives in $I$ and using   the condition $\beta<\fr{1}{\alpha}$ and Young's inequality, we obtain by direct calculation that at the maximum point $(p_0,t_0)$
\begin{align*}I =  -\beta\Big (\fr{1}{\alpha}-\beta\Big ) \fr{u_s^2}{u^2} +2\fr{\eta_s^2}{\eta^2}-  \Big (1-\fr{1}{\alpha}+2\beta\Big )\fr{u_s}{u}\fr{\eta_s}{\eta}\le  -\fr{\beta}{2}\Big (\fr{1}{\alpha}-\beta\Big ) \fr{u_s^2}{u^2} + C \fr{\eta_s^2}{\eta^2},
\end{align*} 
for some   $C=C(\alpha,\beta)>0$. This implies that there exist $C>0$ and $\delta>0$ which depend on $\alpha$ and  $\beta$ such that, at $(p_0,t_0)$, 
\begin{align}\label{eq-prop11}
\delta u^{1-\fr{1}{\alpha}} \fr{u^2_s}{u^2} = \delta \fr{u^2_s}{u^{2\beta}} u^{2\beta-1-\fr{1}{\alpha}} \le \fr{1}{t_0}+ \frac{\square \eta}{\eta}  +(2\alpha+2-\beta\alpha)u^{1+\frac{1}{\alpha}}  +C u^{1-\frac{1}{\alpha}}\fr{\eta_s^2}{\eta^2}.
\end{align}
Next, we define the cut-off function $\eta$ by
\begin{align*}
\eta=\big(1-|X(p,t)|^2\big)_+,
\end{align*}
and observe that we have $|\eta_s| \leq 2$, $|\langle X,n\rangle| \leq |X| \leq 1$ and
\begin{align}\label{eq-cutoffest}
|\square \eta|=\Big |\p_t \eta - \alpha u^{1-\fr{1}{\alpha}} \p_{ss} \eta \Big | =\Big |\alpha u^{1-\fr{1}{\alpha}} (2 + 2(1-\alpha^{-1})u^{\fr{1}{\alpha}}\la X,n\ra ) \Big |\leq C u^{1-\frac{1}{\alpha}},
\end{align}
for some $C=C(\alpha,L)$. Since $t_0\le T$, $0\le\eta\le 1$ and $|\eta_s|\le 2$, at the point $(p_0,t_0)$  we have 
\begin{align*}
\fr{u^2_s}{u^{2\beta}}t_0^2\eta^2 &\le C'\Big (T u^{1+\fr{1}{\alpha}-2\beta}  + u^{2+\fr{2}{\alpha}-2\beta} T^2+ u^{2-2\beta}T^2\Big )\le C' T^2(1+T^{-1}).
\end{align*} 
Here $C'= C'(\alpha,\beta, L)>0$ and this is possible because $\beta \le\min(\alpha^{-1},1)$. 
At any point $(p,T)$, \[\fr{u_s^2}{u^{2\beta}}\eta^2\le \fr{w(p_0,t_0)}{T^2}\le C' (1+ T^{-1}).\]
Therefore, replacing $T$ by $t$ yields the desired result.
\end{proof}
 \end{prop}

 \begin{prop}\label{prop-lowerbddfast} For $\alpha\in(0,1)$, for each $t>0$, \[\liminf_{|X|\to \infty} |X|^{2} \, \bar\ka^{1-\alpha}=\liminf_{|X|\to \infty} |X|^{2}\, u^{\fr{1}{\alpha}-1} \ge \fr{ 2\alpha(1+\alpha)}{1-\alpha}t . \] Here, $|X|=|X(p,t)|$ is the extrinsic distance from the origin. More generally, this is uniform in $t$ for all compact time interval which is away from $t=0$. In other words, for $0<\tau_1<\tau_2<\infty$ and $\e>0$, there is $R>0$ such that \[ |X(p,t)|^{2}\, \bar\ka^{1-\alpha}(p,t)\ge \fr{ 2\alpha(1+\alpha)}{1-\alpha}t-\e\] for all $(p,t)$ with $t\in[\tau_1,\tau_2]$ and $|X(p,t)|\ge R$.
 \begin{proof}We will follow the idea and the proof of \cite{HP} Theorem 2.4, where the same inequality is shown for a solution of the Euclidean fast diffusion equation $w_t =\Delta w^\alpha$.  Recall that the curvature $\bar \kappa$ satisfies the equation 
$\bar \kappa_t = (\bar \kappa)_{ss} + \bar \kappa^\alpha$.  Since here we are  on a   Riemannian manifold (though it's 1D) and the metric is changing with respect to time, we need to modify the proof.  
 
 Let us define the  constant $b$ by $b^{-1}:= \fr{ 2\alpha(1+\alpha)}{1-\alpha}$ and 
 consider then  function $U:\R\times (0,\infty)\to \R$ defined by \[U(x,t)={t^{-\fr{1}{\alpha+1}}}(1+ b |x|^2t^{-\fr{2}{\alpha+1}})^{-\fr{1}{1-\alpha}}.\] Then, it can be directly checked that \[U_\mu(x,t):=\mu^{\fr{2}{1-\alpha}}U(t,\mu x)={t^{-\fr{1}{\alpha+1}}}(\mu^{-2}+ b |x|^2t^{-\fr{2}{\alpha+1}})^{-\fr{1}{1-\alpha}}\] are solutions of the 1D fast diffusion equation $f_t = (f^{\alpha})_{xx}$ for all parameters $\mu>0$. 
 
 \smallskip
\noindent {\em Case 1:} Assume first that  our solution of the $\alpha$-CSF is smooth for $t\ge0$ and has the positive and bounded curvature $0<\ka^\alpha\le L$. Pick a point $p\in N$. Then the intrinsic distance function $s_p(q,t)=\operatorname{dist}_{g(t)}(p,q) \ge0$ is smooth away from $p$ for $t\ge0$. Moreover, $ds=\sqrt {g_{11}} dx$ implies that $\fr{\p}{\p t} ds = \fr{-2\bar \ka^{\alpha+1} g_{11}}{2\sqrt{g_{11}}}dx = -\bar \ka^{\alpha+1} ds$.
Thus, for a curve $\gamma:[0,1]\to M$ which joins $p$ to $q$, we have\[\fr{\p}{\p t} s_p(q,t) = \int_0^1 -\bar\ka^{\alpha+1}|\dot \gamma(\tau)|d\tau 
  \ge \int -\bar \ka^{\alpha} |d\theta| \ge -L\pi.\]
Define $\bar U_\mu :M\times (0,\infty)\to \R$ by \[\bar U_\mu (q,t) = U_\mu(s_p(q,t)+L\pi t, t)= U_\mu (\bar s_p(q,t),t),
\qquad \bar s_p(q,t):= s_p(q,t)+L\pi t.  \]
One can easily check using the  chain rule and the fact that $\p_t \bar s_p =\p_t s_p +L\pi\ge0$ and $\p_x U_\mu \le 0$, that \[\p_t \bar U_\mu = \p_t \bar s_p\p_xU_\mu+ \p_t U_\mu \le \p_t U_\mu =(U_\mu^\alpha)_{xx}= (\bar U_\mu^\alpha)_{ss}. \] i.e. $\p_t\bar U_\mu  -(\bar U^\alpha_\mu)_{ss} \le 0$ away from the non-smooth point $p$. From this point, we can follow the proof of Theorem 2.4 \cite{HP} using these barriers $\bar U_\mu$. 
Let us choose two different points $p_1$, $p_2\in N$ such that $s_p(p_i,0)=1$.  For a fixed $T>0$, let's denote $\delta = \delta(T) >0$ by \[\delta = \fr{1}{2}\min_{i=1,2,\,t\in[0,T]}\, \bar\ka(p_i,t).\]
We can find  small $\mu>0$ such that \[\bar U_\mu(q,t) \le \delta\quad\mbox{when }0\le t\le T,\,  q\in M\setminus B_{g(0)} (1,p). \]This is possible because \[\bar U_\mu(q,t) = U_\mu( \bar s_p(q,t),t )\le U_\mu(s_p(q,0),t)\le U_\mu(1,t) .\] Recall that $\bar \ka_t = (\bar\ka^{\alpha})_{ss} + \bar\ka^{\alpha+2} \ge (\bar\ka^{\alpha})_{ss}$. 
Since $\bar U_\mu(q,0) = 0$, by the  comparison principle  (c.f. Lemma 3.4 \cite{HP}),  \[\bar U_{\mu} (q,t) \le \bar \ka(q,t)  \quad \mbox{for }0\le t\le T,\,  q\in M\setminus B_{g(0)} (1,p).\]
The proof of Lemma 3.4 \cite{HP} uses the Kato's inequality \[\Delta (f^+)\ge (\Delta f)^+ \quad \mbox{in distribution sense} . \]At each fixed time slice and thus fixed metric, this is again true in our (1D) Riemannian case. Thus the proof actually works in our setting,  thus the   comparison principle holds. 
Therefore,  comparing   with our barrier $\bar U_\mu$ yields  that   for each $0<\tau_1<\tau_2\le T$ and $\e>0$, there is $R>0$ such that   \[ \bar s_p^{{2}} (q,t) \, \bar\ka^{1-\alpha} (q,t) \ge \bar s_p^{{2}} (q,t) \, \bar U^{1-\alpha}_{\mu}(q,t) \ge \Big (tb^{-1}\Big )-\e\]
holds, for all $(q,t)$ with $t\in[\tau_1,\tau_2]$ and $\bar s_p(q,t)\ge R$. Also, observe \[\fr{\bar s_p(q,t)}{|X(q,t)|} = \fr{s_p(q,t)+ L\pi t}{|X(q,t)|} \to 1 \] uniformly for $t\in[0,T]$ as $|X(q,t)|\to \infty$. This follows from   the fact that $M_t$ is convex, it is  located between  to two parallel lines, it is asymptotic to these parallel lines and $|X(p,t)-X(p,0)| \le Lt$. We can choose $T>0$ arbitrary large and repeat the same argument to conclude that the proposition holds, under the extra assumption that $0 < \bar \kappa^\alpha \leq L$. 

\smallskip

\noindent {\em Case 2:} For a general solution of the $\alpha$-CSF which is not smooth up to $t=0$ or does not satisfy the  curvature bound $0<\bar\ka^\alpha\le L$, we may apply the previous proof on $t\in [\tau,\infty)$,  for small fixed $\tau>0$ and conclude that for $0<\tau<\tau_1<\tau_2<\infty$ and $\e>0$, there is $R>0$ such that \[ |X(p,t)|^{2}\, \bar\ka^{1-\alpha}(p,t)\ge b^{-1}(t-\tau)-\e\] for all $(p,t)$ with $t\in[\tau_1,\tau_2]$ and $|X(p,t)|\ge R$. We may  chose  $\tau$ small enough so that $b^{-1}\tau \le \e/2$,  finishing  the proof.\end{proof}
 \end{prop}

In the range of exponents $\alpha \in (0,1)$ we have  the following  global and somewhat   improved estimate than Proposition \ref{prop-us}.
 \begin{prop}\label{prop-usfast} For a fixed $\alpha \in (0,1)$, suppose   $0<u:=\bar \ka ^\alpha\le L $,   for $t\ge 0$. Then, there exists some $C=C(\alpha,L,M_0)$ such that \[|{u_s}| 
\le C(1+ \frac 1{\sqrt{t}}) \, u^{\fr{1}{2}\big (1+\fr{1}{\alpha}\big )} \quad \] holds for $t>0$.
\begin{proof} {Given $T$ and $\tau$ with $0<\tau<T$}, we set  $w:=(t-\tau)^2\eta^2 u_s^2u^{-2\beta} $ for $t \in (\tau,T]$ with the fixed exponent $\beta:= \fr{1}{2}\big (1+\fr{1}{\alpha}\big )$ and a smooth cut off function $\eta$. We are going to choose $\eta$ in a different way to use the asymptotic bound in Proposition \ref{prop-lowerbddfast}. Let us fix a usual cut off function $\xi : [0,\infty) \to \R$ such that \[0\le\xi\le1,\quad \xi=1\mbox{ on }[0,1/2],\quad\operatorname{supp}\xi \subset [0,1],\quad 0\le-\xi',\, |\xi''|\le C.\]
Define $\eta(p,t) := \xi \big ( \frac{|X(p,t)|}R \big )$ for $R\gg 1$. Then the following holds by direct computation. 
\begin{claim}\label{claim-3}
\bee
\eta_s^2 \leq \frac{C}{R^2}, \qquad  \quad \square \eta \coloneqq  \p_t \eta - \alpha u^{1-\fr{1}{\alpha}} \p_{ss} \eta  \le C \, \Big ( \frac{u^{1-\frac{1}{\alpha}}}{R^2}+1\Big )
\eee
for some $C=C(\alpha, L, M_0)$.
\begin{proof}[Proof of Claim \ref{claim-3}] We begin by observing that  $\p_s|X|= \fr{|X|^2_s}{2|X|} = \la \fr{X}{|X|},     \p_s X \ra $
implying  $  |\p_s|X||\le 1$, thus   $\eta_s^2=\Big ( \fr{|X|_s}{R}\xi'\Big )^2 \le R^{-2}C$ yields  the first estimate.  Next, by
the chain rule and \eqref{eqn-X2}, we compute 
 \[\ba \square \eta=(\p_t-\alpha u^{1-\fr{1}{\alpha}} \p_{ss})\eta&= \fr{\xi'}{R} (\p_t-\alpha u^{1-\fr{1}{\alpha}} \p_{ss})|X| -  \alpha u^{1-\fr{1}{\alpha}}\fr{\xi''}{R^2}(\p_s|X|)^2\\
&= \alpha u^{1-\fr{1}{\alpha}} \fr{\xi'}{R} \, \Big ( \fr{ (\p_t-\alpha u^{1-\fr{1}{\alpha}} \p_{ss})|X|^2}{\alpha u^{1-\fr{1}{\alpha}}\, 2|X|} + \fr{(\p_s(|X|^2))^2}{4(|X|^2)^{3/2}}  \Big ) -  \alpha u^{1-\fr{1}{\alpha}}\fr{\xi''}{R^2}(\p_s|X|)^2\\
&=\alpha u^{1-\fr{1}{\alpha}} \Big ( \fr{\xi'}{R}\fr{(\p_s|X|)^2-1+(\alpha^{-1} -1) \bar \ka \la X, n\ra }{|X|}- \fr{\xi''}{R^2}(\p_s|X|)^2 \Big ). \ea\]
Note that  since  since $\xi'(|X|/R) =0$ for $|X|/R<\frac{1}{2}$, we have  ${\displaystyle  \fr{|\xi'|}{R}\fr{1}{|X|}\leq  \fr{2|\xi'|(|X|/R)}{R} \fr{1}{|X|}\le C R^{-2}}$. Therefore, $\bar \kappa =u^{\frac{1}{\alpha}}$, $|\xi''|\leq C$  and $\langle X,n \rangle \leq |X|$ yield
\begin{align*}
\square \eta \leq C\frac{u^{1-\frac{1}{\alpha}}}{R^2} +(1-\alpha)u^{1-\frac{1}{\alpha}} \, \bar  \kappa\,  \frac{\xi'}{R}\frac{\langle X,n \rangle}{|X|}\leq C\, \frac{u^{1-\frac{1}{\alpha}}}{R^2}+(1-\alpha) \, u \frac{|\xi'|}{R}.
\end{align*}
By using $u \leq L$, $R \geq 1$, $|\xi'| \leq C$, we obtain the second estimate in the claim.
\end{proof}
\end{claim}

We will now continue with the proof of the proposition.  Assume that  a nonzero maximum of $w(p,t)$ on $t\in[\tau,T]$ is obtained at $(p_0,t_0)$ with $t_0\in (\tau,T]$. 
Since the proof of Proposition \ref{prop-us} does not make use of the specific $\eta$ until \eqref{eq-prop11},  except that it has a compact support, we may use the calculation in  \eqref{eq-prop11} and combine  it  with the above claim 
to conclude that at the point  $(p_0,t_0)$ we have 
\begin{align*}
\delta u^{1-\fr{1}{\alpha}} \fr{u^2_s}{u^2} = \delta \fr{u^2_s}{u^{2\beta}} \le \fr{1}{t_0-\tau}+\frac{C}{\eta}\Big (\frac{u^{1-\frac{1}{\alpha}}}{R^2}+1 \Big ) +C  u^{1+\fr{1}{\alpha}} +\frac{C}{\eta^2}\frac{u^{1-\frac{1}{\alpha}}}{R^2}, 
\end{align*} for some $C=C(\alpha,L,M_0)$. Therefore,  multiplying  the last inequality by $(t_0-\tau)^2 \eta(p_0,t_0)^2$ and using $0<t_0-\tau\le t_0\le T$, $\eta\le1$, $\beta:= \frac 12 \, \big (1 +\frac 1\alpha \big ) $  yield   that for $t \in [\tau,T]$ and $|x| \leq R/2$  the following holds
\begin{align*}
\fr{u^2_s}{u^{2\beta}}(t-\tau)^2 =: w(p,t) \leq w(p_0,t_0):=  \fr{u^2_s}{u^{2\beta}} \eta^2(t_0-\tau)^2 \le  C\, T\Big (1+ T+ \sup_{\operatorname{supp} \eta,\, t=t_0}\fr{u^{1-\fr{1}{\alpha}}}{R^2}t_0 \Big ),
\end{align*}
where $C=C(\alpha,L,M_0)$ but independent of $R>0$ and $\tau>0$. 

Now, we apply Proposition \ref{prop-lowerbddfast} with $\tau_1=\tau$, $\tau_2=T$, $\epsilon=\frac{\alpha(1+\alpha)}{1-\alpha}t_0 $, which implies that there  exists some $R_0>0$ such that $|X|^2u^{\frac{1}{\alpha}-1} \geq \frac{\alpha(1+\alpha)}{1-\alpha}t_0$ for $|X|(p,t_0) \geq R_0$. Combining this with the above estimate yields that   for $R \geq R_0$
\begin{align*}
 \sup_{R_0\leq |X| < R}\fr{u^{1-\fr{1}{\alpha}}(p,t_0)}{R^2}t_0 \leq   \sup_{R_0\leq |X| < R}\frac{t_0}{|X|^2u^{\frac{1}{\alpha}-1}(p,t_0)}\leq \frac{1-\alpha}{\alpha(1+\alpha)}. 
\end{align*}
We conclude that  if $t \in [\tau,T]$ and $|x| \leq R/2$ with $R \geq R_0$ then we have
\begin{align*}
\fr{u^2_s}{u^{2\beta}}(t-\tau)^2 \leq CT\Big (1+ T+  KT R^{-2}\Big ), \quad \text{where}\; K=\sup_{|x|(p,t_0) \leq R_0}u^{1-\frac{1}{\alpha}}(p,t_0).
\end{align*}
Passing  $\tau \to 0$ and $R \to +\infty$ and then setting   $t=T$, we finally obtain the bound 
\begin{align*}
\fr{u^2_s}{u^{2\beta}}T^2 \leq C\, T^2 (1+ T^{-1}),
\end{align*}
which holds for all  $T>0$. By replacing $T$ by $t$, we have the desired result.
\end{proof}
 \end{prop}
 
 \begin{prop}\label{prop-uss} Let $\alpha\in[\fr{1}{3},\infty)$ be fixed. If $|u_s|\le K u^\beta$ for some $\beta \in (0,\fr{1}{\alpha})$ and $ L\geq \sup_{B_1} u$ for $t\geq 0$, then for every $\e \in(0, \min(1,\beta))$ we have  \[|u_{ss}| \le C(1+t^{-1/2})\, u^{{\min(\beta,2\beta-1)-\e}} \quad\mbox{ on } \, B_{1/2}  \] with some $C=C(\alpha,\beta,\e,K,L).$
 \begin{proof}
Let us define $  v := u^{1-\beta}(u_s+ 2Ku^\beta)= u^{1-\beta}u_s + 2Ku$.
 Note that \[Ku\le   v \le 3Ku.\]
 \begin{claim}\label{claim-solong} There is $C'=C'(\alpha,\beta,K)>0$ such that
\[\fr{|(\p_t - \alpha u^{1-\fr{1}{\alpha}} \p_{ss})  v |}{\alpha u^{1-\fr{1}{\alpha}}}\le C' \Big ({|  v_s|} u^{\beta-1}+   vu^{2\beta-2} +   vu^{\fr{2}{\alpha}} \Big ) \]and
\bee
\fr{|(\p_t - \alpha u^{1-\fr{1}{\alpha}} \p_{ss})  v_s |}{\alpha u^{1-\fr{1}{\alpha}}}\le C'\Big (|  v_s|\bigl(\fr{|  v_{ss}|}{|  v_s|}+ \fr{|  v_s|}{  v}\bigr) u^{\beta-1}+|  v_s|( u^{2\beta-2} + u^{\fr{2}{\alpha}})+  v(u^{2\beta}+u^{\beta-1+\fr{2}{\alpha}}+u^{3\beta-3}) \Big ).
\eee\end{claim}
Since a proof of this claim is   long, let us postpone it for the end of this proposition and assume   it is true.  Define $w:=   v_s^2   v^{-2\gamma}t^2\eta^2$ for some $0<\gamma<1$ to be determined later, where $\eta=(1-|X(p,t)|)^2_+$. Then, on the support $\{w>0\}$ we have
\begin{align*}
\frac{\p_t w}{2w}=\frac{1}{t}+\fr{\p_t\eta}{\eta}+\frac{\p_t  v_s}{  v_s}-\gamma\frac{\p_t   v}{  v}\quad
\mbox{ and }\quad
\frac{\p_s w}{2w}=\frac{\p_s\eta}{\eta}+\frac{\p_s  v_s}{  v_s}-\gamma \frac{\p_s   v}{  v}
\end{align*}
and 
\begin{align*}
\frac{\p_{ss} w}{2w}-\frac{w_s^2}{2w^2}=\frac{\eta_{ss}}{\eta}-\frac{\eta_s^2}{\eta^2}+\frac{\p_{ss}  v_s}{  v_s}-\frac{  v_{ss}^2}{  v_s^2}-\gamma\frac{v_{ss}}{  v}+\gamma\frac{  v_s^2}{  v^2}.
\end{align*}
Suppose that a nonzero maximum of $w$ on $t\in[0,T]$ is attained at $(p_0,t_0)$. At this point, 
\be\label{eqn-form1} 
\ba0 &\le \frac{\p_t w}{2w} - \alpha u^{1-\frac{1}{\alpha}}\frac{\p_{ss} w}{2w}+\alpha u^{1-\frac{1}{\alpha}}\frac{w_s^2}{2w^2}\\
& =\fr{1}{t_0}+\alpha u^{1-\fr{1}{\alpha}}\Big (\fr{(\p_t-\alpha u^{1-\fr{1}{\alpha}}\p_{ss})\eta_{ss}}{\alpha u^{1-\fr{1}{\alpha}}\eta} 
 +\fr{(\p_t - \alpha u^{1-\fr{1}{\alpha}} \p_{ss})  v_s }{\alpha u^{1-\fr{1}{\alpha}}  v_s} -\gamma \fr{(\p_t - \alpha u^{1-\fr{1}{\alpha}} \p_{ss})  v }{\alpha u^{1-\fr{1}{\alpha}}  v}\Big )\\
&\quad+\alpha u^{1-\fr{1}{\alpha}}\Big (\fr{  v_{ss}^2}{  v_s^2} -\gamma \fr{  v_s^2}{  v^2} +\fr{\eta_s^2}{\eta^2}\Big ).\ea 
\ee
Since $w_s=0$ at the maximum $(p_0,t_0)$, we have
\[\ba\fr{  v_{ss}^2}{  v_s^2} - \gamma \fr{  v_s^2}{  v^2} +\fr{\eta_s^2}{\eta^2} &=(-\gamma+\gamma^2)\fr{  v_s^2}{  v^2}+2\fr{\eta_s^2}{\eta^2} - 2\gamma \fr{\eta_s}{\eta } \fr{  v_s}{  v}\le -\fr{\gamma(1-\gamma)}{2} \fr{  v_s^2}{  v^2} +C\fr{\eta_s^2}{\eta^2} \ea\]and  by using $\gamma <1$
\[\Big |\fr{  v_{ss}}{  v_s}\Big |=\Big |\gamma \fr{  v_s}{  v} - \fr{\eta_s}{\eta}\Big |\le  \fr{  |v_s|}{  v} +  \fr{|\eta_s|}{\eta}.\]
Also, recall  \eqref{eq-cutoffest}. Then \eqref{eqn-form1} together with  the Claim, \eqref{eq-cutoffest} and
the last two estimates above yield 
\[\ba0 &\le\fr{1}{t_0}+\alpha u^{1-\fr{1}{\alpha}}\Big (\fr{C}{\eta} +\fr{|(\p_t - \alpha u^{1-\fr{1}{\alpha}} \p_{ss})  v_s |}{\alpha u^{1-\fr{1}{\alpha}}|  v_s|}+ \fr{|(\p_t - \alpha u^{1-\fr{1}{\alpha}} \p_{ss})  v| }{\alpha u^{1-\fr{1}{\alpha}}  v}- \fr{\gamma(1-\gamma)}{2} \fr{  v_s^2}{  v^2} +C\fr{\eta_s^2}{\eta^2}\Big )\\
&\le\fr{1}{t_0} + C u^{1-\fr{1}{\alpha}}\Big (  \fr{1}{\eta}   + \fr{|  v_s|}{  v}u^{\beta-1}+ \fr{|\eta_s|}{\eta} u ^{\beta-1}+u^{2\beta-2} + u^{\fr{2}{\alpha}}\Big . \\&\quad\quad\quad\quad\quad\quad\quad\quad\Big .+\fr{  v}{|  v_s|}(u^{2\beta}+u^{\beta-1+\fr{2}{\alpha}}+u^{3\beta-3})+ \fr{\eta_s^2}{\eta^2}\Big )-  \fr{
\alpha\gamma(1-\gamma)}{2} \fr{u^{1-\frac{1}{\alpha}}  v_s^2}{  v^2}\\
&\le\fr{1}{t_0} +C u^{1-\fr{1}{\alpha}}  \Big ( \fr{1}{\eta}  +u^{2\beta-2} + u^{\fr{2}{\alpha}}+\fr{  v}{|  v_s|}(u^{2\beta}+u^{\beta-1+\fr{2}{\alpha}}+u^{3\beta-3})+ \fr{\eta_s^2}{\eta^2}\Big )- \fr{\alpha\gamma(1-\gamma)}{4}\fr{u^{1-\frac{1}{\alpha}}  v_s^2}{  v^2} \\ 
&\le \fr{1}{t_0}+C u^{1-\fr{1}{\alpha}} \Big ( \fr{1}{\eta}  +u^{2\beta-2}+\fr{  v}{|  v_s|}u^{3\beta-3}+ \fr{\eta_s^2}{\eta^2}\Big )- \fr{\alpha\gamma(1-\gamma)}{4}\fr{u^{1-\frac{1}{\alpha}}  v_s^2}{  v^2}   \quad \mbox{(since }\beta<\fr{1}{\alpha}\le3\mbox{)}\ea \]  for some $C=C(\alpha, \beta, \gamma, K,L)$,  where the dependence of $C$ on $L$ takes place in the last inequality for the first time.
We conclude that at the maximum point $(p_0,t_0)$ the following holds
$$\fr{u^{1-\frac{1}{\alpha}}  v_s^2}{  v^2}  \leq C\, \Big (  \fr{1}{t_0}+ u^{1-\fr{1}{\alpha}} \big ( \fr{1}{\eta}  +u^{2\beta-2}+\fr{  v}{|  v_s|}u^{3\beta-3}+ \fr{\eta_s^2}{\eta^2}\big ) \Big ) .$$
Using this estimate we now conclude that 
\[\ba w(p_0,t_0)&=\bigg(\fr{  v_s^2}{  v^2}u^{1-\frac{1}{\alpha}}\bigg) u^{\frac{1}{\alpha}-1}v^{2(1-\gamma)}\,  t_0^2\, \eta^2\\ &\le C \Big (t_0 u^{\fr{1}{\alpha}-1} \eta^2+ t_0^2\eta^2    v^{2(1-\gamma)}\Big ( \fr{1}{\eta} + \fr{\eta_s^2}{\eta^2} +u^{2(\beta-1)}+\fr{  v}{|  v_s|}u^{3(\beta-1)}\Big ) \Big ) \\ \quad&\le C\Big ( Tu^{1+\fr{1}{\alpha}-2\gamma}+ T^2(u^{2(1-\gamma)}+u^{2(\beta-\gamma)})+ \fr{  v^{\gamma}}{|  v_s|t_0\eta} u^{3(\beta-\gamma)} T^3 \Big )\quad \mbox{(since }   v \le 3Ku \mbox{)}. \ea\]
If we choose our $\gamma\in(0,1)$ by $\gamma := \min(1,\beta) -\e$, then $1+\fr{1}{\alpha}-2\gamma$, $1-\gamma$, $\beta-\gamma\ge0$ as $\gamma<1$ and $\gamma<\beta<\fr{1}{\alpha}$. Thus 
\[w(p_0,t_0)\le C\Big ( T+ T^2+ \fr{1}{w^{1/2}(p_0,t_0)}T^3 \Big )\] with $C=C(\alpha, \beta,\e ,K,L)$.
By considering the  two cases $w(p_0,t_0)\ge T^2$ and $<T^2 $, we finally obtain the bound \[w(p_0,t_0)\le CT^2(1+T^{-1}) \]
implying that  at any point $(p,T)$, we have  \[\fr{|  v_s|}{u^\gamma}\, \eta  \le (3K)^\gamma\fr{|  v_s|}{ v^\gamma} \, \eta= (3K)^\gamma\fr{w^{1/2}(p,t)}{T}  \le C(1+T^{-1/2}) .\]
Note that $ v_s = u^{1-\beta} u_{ss} + (1-\beta)u^{-\beta} u_s^2 + 2Ku_s$. Hence, $|u_s|\leq Ku^{\beta}$ leads to  
\[\fr{u^{1-\beta} |u_{ss}|}{u^\gamma}\eta = \fr{|u_{ss}|\eta}{u^{\gamma-(1-\beta)}}  \le C(1+T^{-1/2}) + C \fr{u^{-\beta}u_s^2 + |u_s|}{u^{\gamma}} \le C(1+T^{-1/2})   .\]
We replace $T$ by $t$. Then,  $u^{\gamma-(1-\beta)} =  u^{\min(\beta,2\beta-1)-\e}$ yields the proposition.

\begin{proof}[Proof of Claim \ref{claim-solong}] During the proof of the claim we will frequently  use 
the inequalities $$Ku \le v \le 3K u \qquad \mbox{and} \qquad |u_s|\le K\, u^\beta$$ 
and we will denote by  $C$ various constants which depend  on $\alpha$, $\beta$ and $K$. 

Since
\[\ba(\p_t-\alpha u^{1-\alpha} \p_{ss})v&=  u_{s}(\p_t-\alpha u^{1-\alpha} \p_{ss})u^{1-\beta} + u^{1-\beta} (\p_t-\alpha u^{1-\alpha} \p_{ss})u_{s} -2\alpha u^{1-\fr{1}{\alpha}} (u^{1-\beta})_s( u_{s})_s \\
&\quad +2K(\p_t-\alpha u^{1-\alpha} \p_{ss})u\ea\]
to  show the first inequality in the claim, it is enough the  terms on the right hand side by  
\begin{align*}
C  u^{1-\fr{1}{\alpha}} \Big ({|  v_s|} u^{\beta-1}+   vu^{2\beta-2} +   vu^{\fr{2}{\alpha}}\Big  ) \leq Cu^{1-\fr{1}{\alpha}} \Big ({|  v_s|} u^{\beta-1}+   u^{2\beta-1} +   u^{\fr{2}{\alpha}+1} \Big ).
\end{align*}
We begin with observing\be \label{eq-vs} v_s = u^{1-\beta} u_{ss} + (1-\beta) u^{-\beta} u_s^2 +2K u_s \ee and thus $|u_s|\leq Ku^{\beta}$ and $Ku \le v$ yield \be \label{eq-uss2}|u_{ss}|\le u^{\beta-1}|v_s|+Cu^{-1}u_s^2+Cu^{\beta-1}|u_s| \le C(|v_s| u^{\beta-1}+ u^{2\beta-1} ).\ee 
Therefore, 
\begin{align}\label{eq-ues}
\fr{|(\p_t-\alpha u^{1-\alpha} \p_{ss})u|}{\alpha u^{1-\fr{1}{\alpha}}}  = \big|u_{ss} + u^{1+\fr{2}{\alpha}}\big| \le C\Big ({|  v_s|} u^{\beta-1}+   u^{2\beta-1} +   u^{\fr{2}{\alpha}+1} \Big ) 
\end{align}
also implying that 
\begin{align}\label{eq-uest}
u_{s}\fr{(\p_t-\alpha u^{1-\alpha} \p_{ss})u^{1-\beta}}{\alpha u^{1-\fr{1}{\alpha}}} &=  (1-\beta)u_s u^{-\beta}\fr{(\p_t-\alpha u^{1-\alpha} \p_{ss})u}{\alpha u^{1-\fr{1}{\alpha}}}-\alpha(1-\beta)(-\beta) u^{-1-\beta}u_s^3 \notag \\
& \leq C\Big ({|  v_s|} u^{\beta-1}+   u^{2\beta-1} +   u^{\fr{2}{\alpha}+1}\Big ).
\end{align}
In addition,  using \eqref{eq-uss2} we have 
\be \label{eq-usest}\ba u^{1-\beta} \fr{|(\p_t-\alpha u^{1-\alpha} \p_{ss})u_{s}|}{\alpha u^{1-\fr{1}{\alpha}}} \le Cu^{1-\beta} (\fr{|u_s u_{ss}|}{u} +u^{\fr{2}{\alpha}}|u_s|)\le C\Big ({|  v_s|} u^{\beta-1}+    u^{2\beta-1} +   u^{\fr{2}{\alpha}+1} \Big ) \ea \ee
and
\[|(u^{1-\beta})_s( u_{s})_s |= |(1-\beta)u^{-\beta} u_s u_{ss} |\le C |u_{ss}|\le C\Big ({|  v_s|} u^{\beta-1}+   u^{2\beta-1}  \Big ).\] 
Combining the above inequalities yieds the first estimate in the claim. 
\medskip

Next, by using \eqref{eq-vs} we compute
\[\ba &(\p_t -\alpha u^{1-\fr{1}{\alpha}} \p_{ss}) v_s=   2K (\p_t -\alpha u^{1-\fr{1}{\alpha}} \p_{ss})u_s\\&\quad+  u_{ss}(\p_t -\alpha u^{1-\fr{1}{\alpha}} \p_{ss})u^{1-\beta}-2\alpha u^{1-\fr{1}{\alpha}} (u^{1-\beta})_s (u_{ss})_s + u^{1-\beta} (\p_t -\alpha u^{1-\fr{1}{\alpha}} \p_{ss})u_{ss}  \\
&\quad +(1-\beta)\Big (u_s^2(\p_t -\alpha u^{1-\fr{1}{\alpha}} \p_{ss})u^{-\beta}-2\alpha u^{1-\fr{1}{\alpha}} (u^{-\beta})_s (u_{s}^2)_s + u^{-\beta} (\p_t -\alpha u^{1-\fr{1}{\alpha}} \p_{ss})u_{s}^2\Big ). \ea\]
To show the second inequality in the claim, we bound the seven terms above by $C  u^{1-\fr{1}{\alpha}} H$ where
\begin{align*}
H\coloneqq &u^{\beta-1}\Big (   |  v_{ss}|  +  |  v_s|^2 u^{-1}+|  v_s|( u^{\beta-1} + u^{1+\fr{2}{\alpha}-\beta})+  u^{\beta+2}+u^{\fr{2}{\alpha}+1}+u^{2\beta-1} \Big ) \\
\geq &\Big (|  v_s|\bigl(\fr{|  v_{ss}|}{|  v_s|}+ \fr{|  v_s|}{  v}\bigr) u^{\beta-1}+|  v_s|( u^{2\beta-2} + u^{\fr{2}{\alpha}})+  v(u^{2\beta}+u^{\beta-1+\fr{2}{\alpha}}+u^{3\beta-3}) \Big ).
\end{align*}
The inequality \eqref{eq-usest} implies that the first term is bounded \[ \fr{|(\p_t-\alpha u^{1-\alpha} \p_{ss})u_{s}|}{\alpha u^{1-\fr{1}{\alpha}}}\le C H. \] 
To proceed, we observe
\begin{align}\label{eq-H property}
u^{-\beta}\Big ( |v_s| u^{\beta-1}+ u^{2\beta-1} \Big )\Big ({|  v_s|} u^{\beta-1}+   u^{2\beta-1} +   u^{\fr{2}{\alpha}+1} \Big )\le CH.
\end{align} 
Hence, by using \eqref{eq-uss2} and \eqref{eq-uest} we estimate the second term, \[ \ba \fr{|u_{ss}(\p_t-\alpha u^{1-\alpha} \p_{ss})u^{1-\beta}|}{\alpha u^{1-\fr{1}{\alpha}}}\le CH.\ea\] 
Now, we differentiate \eqref{eq-vs} again so that we have
\[v_{ss}= u^{1-\beta}u_{sss}+3(1-\beta)u^{-\beta}u_s u_{ss} +(1-\beta)(-\beta)u^{-1-\beta}u_s^3 + 2Ku_{ss}. \] 
Thus, by using \eqref{eq-uss2} we have
\begin{align}\label{eq-usssest}
|u_{sss}|=& u^{\beta-1}\Big | v_{ss} -3(1-\beta)u^{-\beta}u_s u_{ss} -(1-\beta)(-\beta)u^{-1-\beta}u_s^3 + 2K u_{ss} \Big | \notag\\
\le & u^{\beta-1}|v_{ss}|+Cu^{\beta-1}|u_{ss}|+Cu^{3\beta-2}\leq CH.
\end{align} 
Hence, we can bound the third term, as  follows \[ \ba|(u^{1-\beta})_s (u_{ss})_s|\le C u^{-\beta}|u_su_{sss}|\le C |u_{sss}|\le CH.\ea\] 
We recall \eqref{eq-uss} to estimate the fourth term \[ \ba \fr{|u^{1-\beta}(\p_t-\alpha u^{1-\alpha} \p_{ss})u_{ss}|}{\alpha u^{1-\fr{1}{\alpha}}}\le Cu^{1-\beta} \Big (\fr{|u_s u_{sss}|}{u}+ \fr{u_{ss}^2}{u} + \fr{u_s^2|u_{ss}|}{u^2} +u^{-1+\fr{2}{\alpha}} u_s^2   +u^{\fr{2}{\alpha}}|u_{ss}| \Big ).\ea\]
This combined with \eqref{eq-uss2}, \eqref{eq-H property}, and \ref{eq-usssest} yields
\[ \ba \fr{|u^{1-\beta}(\p_t-\alpha u^{1-\alpha} \p_{ss})u_{ss}|}{\alpha u^{1-\fr{1}{\alpha}}}\le C|u_{sss}|+Cu^{\beta+\frac{2}{\alpha}}+Cu^{-\beta}|u_{ss}| \Big ( |u_{ss}|  + u^{2\beta-1}    +u^{\fr{2}{\alpha}+1} \Big )\leq CH .\ea\]
The fifth term is
\[u_s^2\fr{(\p_t -\alpha u^{1-\fr{1}{\alpha}} \p_{ss})u^{-\beta} }{\alpha u^{1-\alpha}}= (-\beta)u_s^2 u^{-\beta-1} \fr{(\p_t -\alpha u^{1-\fr{1}{\alpha}} \p_{ss})u}{\alpha u^{1-\alpha}}-\beta (1+\beta) u^{-\beta-2} u_s^4.\]
Therefore, by using \eqref{eq-ues} we have
\[u_s^2\fr{|(\p_t -\alpha u^{1-\fr{1}{\alpha}} \p_{ss}) u^{-\beta}| }{\alpha u^{1-\alpha}}\leq C  u^{\beta-1} \fr{|(\p_t -\alpha u^{1-\fr{1}{\alpha}} \p_{ss})u|}{ u^{1-\alpha}}+C u^{3\beta-2}  \leq CH.\]
The sixth term is bounded by \eqref{eq-uss2}
 \[|(u^{-\beta})_s (u_{s}^2)_s| = |-2\beta u^{-1-\beta}u_s^2u_{ss}|\le CH .\]
The last seventh term 
\begin{align*}
u^{-\beta}\fr{|(\p_t -\alpha u^{1-\fr{1}{\alpha}} \p_{ss})u_s^2 }{\alpha u^{1-\alpha}}= &\Big  |2u_s u^{-\beta} \fr{(\p_t -\alpha u^{1-\fr{1}{\alpha}} \p_{ss})u_s}{\alpha u^{1-\alpha}}-2u^{-\beta} u_{ss}^2\Big | \\
\le & C \fr{|(\p_t -\alpha u^{1-\fr{1}{\alpha}} \p_{ss})u_s|}{ u^{1-\alpha}}+Cu^{-\beta}u_{ss}^2.
\end{align*}
By using \eqref{eq-usest} and \eqref{eq-H property} we can estimate the first term above
\begin{align*}
\fr{|(\p_t -\alpha u^{1-\fr{1}{\alpha}} \p_{ss})u_s|}{ u^{1-\alpha}}& \leq Cu^{\beta-1}\Big ({|  v_s|} u^{\beta-1}+   u^{2\beta-1} +   u^{\fr{2}{\alpha}+1} \Big )\\
&= Cu^{-\beta}\Big (u^{2\beta-1}\Big )\Big ({|  v_s|} u^{\beta-1}+   u^{2\beta-1} +   u^{\fr{2}{\alpha}+1} \Big )\leq CH.
\end{align*}
Moreover, \eqref{eq-uss2} and \eqref{eq-H property} show $u^{-\beta}u_{ss}^2 \leq CH$.
We conclude from the above discussion that all seven terms on  $ (\p_t -\alpha u^{1-\fr{1}{\alpha}} \p_{ss}) v_s$
are bounded by $C u^{1-\frac 1\alpha} \, H $, finishing  the proof of the claim.  \end{proof}
 \end{proof}
 \end{prop}

We are finally ready to give the proof of our main estimate which will be used in the next section to control the boundary
terms. Note that while most of our previous estimates hold for $\alpha >0$ or $\alpha > 1/3$, for our estimate below $\alpha > 1/2$ is required. 
 
\begin{thm}\label{thm-decay} Assume that  $M_t$, $t \in [0,+\infty)$ is a solution  of the $\alpha$-CSF with  $\alpha >1/2$ and the initial data $M_0$ satisfying the assumptions of Proposition \ref{prop-basic}. 
Then there exists  $\e=\e(\alpha)>0$ so that  \[ \fr{|(\bar\ka^\alpha)_s(\bar\ka^\alpha)_t|}{\bar\ka}=\fr{|u_s\, u_{t}|}{u^{1/\alpha}}\le C(t_0,M_0) \, u^\e, \qquad \mbox{for}\,\, t > t_0. \]
\begin{proof} By equation \eqref{eqn-speed}, $u:=\bar \kappa^\alpha$ satisfies 
$$\fr{u_s u_t}{u^{1/\alpha}}  = u_s \, \alpha u^{1-\fr{1}{\alpha}} (u_{ss}+u^{1+ {2}/{\alpha}})\, u^{-\fr{1}{\alpha}}=\alpha \, \fr{u_su_{ss}}{u^{\frac 2\alpha -1} } + \alpha   u_s u^2.$$  By Proposition \ref{prop-curvbdd}, we have a uniform upper bound on $u$ for $t\ge t_0>0$ which combined with  Proposition \ref{prop-us} yields  desired bound for the second term. We will next  take care the first term. 

First, suppose $\alpha \in[1, \infty)$. Combining  Proposition \ref{prop-us} and   \ref{prop-uss} together with our  curvature bound (which is assumed in Proposition \ref{prop-uss}), implies that  for $\e\in(0,\fr{1}{\alpha})$, \[|u_su_{ss}| \le C(\alpha,\e,M_0, t_0) \, u^{\fr{3}{\alpha}-1-\e}\quad\text{ for }\,\, t\ge t_0>0,  \]
where we have used that for $\alpha \in[1, \infty)$, $\min (1/\alpha, 2/\alpha-1) = 2\alpha^{-1}-1$. 

When $\alpha \in (1/3,1)$, then $\min (1/\alpha, 2/\alpha-1) =  1/\alpha$, thus   Proposition \ref{prop-usfast}, \ref{prop-uss} and our  curvature bound imply that for  for $0<\e <1$, 
 \[|u_su_{ss}| \le C(\alpha,\e,M_0,t_0) \,  u^{1+\fr{1}{\alpha}-\e}. \]  
 Since $\fr{2}{\alpha}-1 <1+\fr{1}{\alpha}$ iff $\alpha >\fr{1}{2}$, we obtain the desired result for every $\alpha >\fr{1}{2}$.
\end{proof}\end{thm}
\begin{corollary}\label{cor-decay}Under the same conditions as in Theorem \ref{thm-decay} and for any $\alpha>1/2$, there is $\e'(\alpha)>0$ such that 
\[|(\ka^\alpha)_\theta(\ka^\alpha)_t| \le C(t_0,M_0) (\ka^\alpha )^{\e'}\quad \text{ for } t> t_0.\]
\begin{proof} By \eqref{eq-kappas} and Theorem \ref{thm-decay},
\[|(\ka^\alpha)_\theta(\ka^\alpha)_t| = \fr{|(\bar\ka^\alpha)_s|}{\bar\ka}|(\bar\ka^\alpha)_t- \alpha^2 \bar\ka^{2\alpha-3} \bar\ka_s^2|\le C\Big ( (\bar\ka^{\alpha })^\e+ \fr{((\bar\ka^\alpha)_s)^3}{\bar\ka^2}\Big ).\]
In addition, for any $\alpha>1/3$, Propositions  \ref{prop-us} and \ref{prop-usfast} imply that  there is $\e'(\alpha)>0$ such that \[\fr{((\bar\ka^\alpha)_s)^3}{\bar\ka^2}\le C (\bar \ka^\alpha)^{\e'}.\]
\end{proof}
\end{corollary}

\section{Convergence to Translator}

In this final section, we prove our convergence  result Theorem  \ref{thm-goal} from which Theorem \ref{thm-main}
also follows.  The main step   in our  proof   is
 Lemma \ref{lemma-main} which follows from  our decay estimates in the previous section  and an appropriate use of the following entropy.

\begin{definition}
For a strictly convex solution to the  $\alpha$-CSF, we define \[ J^\e(t) :=\fr{(\alpha+1)^2}{\alpha^2}\int_\e^{\pi-\e} (\ka^\alpha)^2_\theta - (\ka^\alpha)^2 \, d \theta \]
which can be also expressed in terms of the pressure function $p:= \kappa^{\alpha+1}$, as 
\[ J^\e(t) =  \int_\e^{\pi-\e} \fr{p_\theta^2}{p^{\fr{2}{\alpha+1}}} - \fr{(\alpha+1)^2}{\alpha^2}p^{\fr{2\alpha}{\alpha+1}}\, d\theta.\] 
Also, set  \[J(t):= \lim_{\e\to0} J^\e(t)\in(-\infty,\infty]\] and this  is well defined due to curvature upper bound in Proposition \ref{prop-curvbdd}. 
\end{definition}

Assume  that $M_t$, $t \in [0,+\infty)$ is a solution of the $\alpha$-CSF which satisfies the assumptions of Theorem \ref{thm-goal}.  We first  observe that $J(t)$ is bounded on $[t_0, +\infty)$, for all $t_0 >0$. 

\begin{lemma}\label{lem-entropybdd} For $\alpha\ge1$, $J(t)\le C(t_0, M_0)<\infty$ for $t\ge t_0>0$. 

\begin{proof}By the evolution of $p=\ka^{\alpha+1}$ given in  \eqref{eq-pressure}, we have 
\be\label{eq-11} \ba \fr{\alpha+1}{\alpha^2}  \int^{\pi-\e}_\e \fr{p_t}{p^{\fr{2}{\alpha+1}}} \, d\theta &= \fr{\alpha+1}{\alpha}  \int_\e^{\pi-\e} p^{\fr{\alpha-1}{\alpha+1}} p_{\theta \theta} -\fr{1}{\alpha}\fr{p_\theta^2}{p^{\fr{2}{\alpha+1}}} +\fr{(\alpha+1)^2}{\alpha^2} p^{\fr{2}{\alpha+1}} d\theta \\
&=-J^\e(t) +\Big (\fr{\alpha+1}{\alpha} p^{\fr{\alpha-1}{\alpha+1}} p_\theta \Big )^{\theta=\pi-\e}_{\theta=\e}. \ea\ee
Note that $p^{\fr{\alpha-1}{\alpha+1}} p_\theta= (\alpha+1)\ka^{2\alpha-2}\ka_s= \fr{(\alpha+1)}{\alpha}\fr{u_s}{u^{\fr{1}{\alpha}-1}}$ and this is uniformly bounded for $t\ge t_0$ when $\alpha>\fr{1}{2}$ in view of Proposition \ref{prop-us}.  In addition,  the Harnarck inequality in  Proposition \ref{prop-harnack} implies, \be \label{eq-13}-\fr{p_t}{p^{\fr{2}{\alpha+1}}}= -\fr{p_t}{p} p^{\fr{\alpha-1}{\alpha+1}}\le \fr{\ka^{\alpha-1}}{t}\ee and therefore 
\[  \int_\e^{\pi-\e} -\fr{p_t}{p^{\fr{2}{\alpha+1}}} \leq  \int_\e^{\pi-\e} \fr{\ka^{\alpha-1}}{t}d\theta\le \int_0^{\pi} \fr{\ka^{\alpha-1}}{t} d\theta .\]
This integrand is uniformly bounded for $\alpha \ge 1$ and $t\ge t_0$. Combining the  above shows that 
$J_\e (t)\le C(t_0, M_0)<\infty$, which implies the desired result. 
%
%{\color{red}When $\alpha\in(1/2,1)$, \[ \int_0^{\pi} \fr{\ka^{\alpha-1}}{t} d\theta= \int_M \fr{\ka^{\alpha}}{t} ds\] and \eqref{eq-ubdd} implies uniform boundedness of this for $1/3<\alpha<1$. No!!}
\end{proof}
\end{lemma}

\begin{prop}Suppose $\alpha\ge1$. 
For $0<t_1<t_2<\infty$, we have  \[J(t_2) -J(t_1) = -\fr{ 2(\alpha+1)^2}{\alpha}\int_{t_1}^{t_2}\int_0^{\pi}  \ka^{\alpha+1} [(\ka^\alpha)_{\theta\theta}- \ka^\alpha]^2 \, d\theta dt. \]
\begin{proof} Since everything is smooth and bounded  on $[\e,\pi-\e]\times[t_1,t_2]$,  we have
\be\label{eq-entropy}\ba \frac{d}{dt}  J^\epsilon(t)  &=\fr{(\alpha+1)^2}{\alpha^2} \int_\epsilon^{\pi-\epsilon} \Big ( 2(\ka^\alpha)_\theta(\ka^\alpha)_{t\theta} -2(\ka^\alpha)(\ka^\alpha)_t \Big ) \, d\theta \\
&= -\fr{2(\alpha+1)^2}{\alpha^2}\int_\epsilon^{\pi-\epsilon} \fr{\ka_t (\ka^\alpha)_t }{\ka^2} \, d\theta + \Big (\fr{2(\alpha+1)^2}{\alpha^2} (\ka^\alpha)_\theta (\ka^\alpha)_t \Big )^{\theta=\pi-\epsilon}_{\theta=\epsilon}\\
& = \int_\epsilon^{\pi-\epsilon} \fr{-2(\alpha+1)^2}{\alpha} \fr{\ka_t^2}{\ka^{3-\alpha}} \, d\theta+ \Big (\fr{2(\alpha+1)^2}{\alpha^2} (\ka^\alpha)_\theta (\ka^\alpha)_t \Big )^{\theta=\pi-\epsilon}_{\theta=\epsilon}\\
& = \fr{2(\alpha+1)^2}{\alpha^2}\Big (-\alpha\int_\epsilon^{\pi-\epsilon} \ka^{\alpha+1} [(\ka^\alpha)_{\theta\theta}+ \ka^\alpha]^2 \, d\theta+ \Big ( (\ka^\alpha)_\theta (\ka^\alpha)_t \Big )^{\theta=\pi-\epsilon}_{\theta=\epsilon} \Big ).
\ea\ee
In view of Theorem \ref{thm-decay}, Theorem \ref{thm-curvaturedecay}, and Lemma \ref{lem-entropybdd}, for $\alpha\ge1$ we can take $\e\to0$ and monotone convergence theorem implies the result.
\end{proof}
\end{prop}
In the case $\alpha\in(1/2,0)$, we cannot  not show that the entropy is finite, so we avoid using the global entropy 
defined on $[0,\pi]$ and approach differently.  Our decay estimate is sufficient  to carry out this,  as we see   in the lemma  below.
\begin{lemma}\label{lemma-main} Assume that $\alpha >1/2$. 
For fixed $\tau>0$  and $\delta>0$, we have \[ \int_{t}^{t+\tau}\int_\delta^{\pi-\delta}  \ka^{\alpha+1} \big ( (\ka^\alpha)_{\theta\theta}+\ka^\alpha \big )^2 d\theta dt \to 0 \qquad\mbox{as}\,\,  t\to \infty.\]
\begin{proof} It suffices to prove that for every $\e>0$, there exist $\bar \delta \in (0,\delta)$ and $t_0>0$ such that 
 \[\int_{t}^{t+\tau}\int_{\bar\delta}^{\pi-\bar\delta}  \ka^{\alpha+1} \big ( (\ka^\alpha)_{\theta\theta}+\ka^\alpha \big )^2 \, d\theta dt \le \e \quad\mbox{for}\quad t\ge t_0.\] 
In view of  \eqref{eq-entropy},  for $0<\bar\delta<\delta$ and $t \geq t_0>0$, we have 
\[\ba \int_{t}^{t+\tau}\int_{\bar\delta}^{\pi-\bar\delta}\ka^{\alpha+1} [(\ka^\alpha)_{\theta\theta}+ \ka^\alpha]^2d\theta dt&=\fr{\alpha}{2(\alpha+1)^2}(J^{\bar\delta}(t)-J^{\bar\delta}(t+\tau))+\fr{1}{\alpha}\int_{t}^{t+\tau}\Big ( (\ka^\alpha)_\theta (\ka^\alpha)_t \Big )^{\theta=\pi-\bar\delta} _{\theta=\bar\delta}dt.\ea\]
First, we  control the boundary terms using  Theorem \ref{thm-curvaturedecay} and  Corollary \ref{cor-decay}
\[\ba\Big |\int_{t}^{t+\tau}\Big ( (\ka^\alpha)_\theta (\ka^\alpha)_t \Big )^{\theta=\pi-\bar\delta} _{\theta=\bar\delta}dt\Big |&\le \tau \sup_{t\in[t_0,\tau]}[(\ka^\alpha)_\theta (\ka^\alpha)_t] (\bar\delta, t) +\tau\sup_{t\in[t_0,\tau]} [(\ka^\alpha)_\theta (\ka^\alpha)_t ](\pi-\bar\delta, t)\\
&\le 2\tau \, C(t_0,\bar\delta,M_0)\quad \mbox{with} \,\,  C(t_0,\bar\delta,M_0)\to 0 \mbox{ as }\bar\delta\to0. \ea\] 
Thus, for given $\e>0$ and $t_0>0$, there exists $\delta_0$ such that if $0<\bar \delta\le \delta_0$ and $t\ge t_0$, \[\ba\Big |\int_{t}^{t+\tau}\Big ( (\ka^\alpha)_\theta (\ka^\alpha)_t \Big )^{\theta=\pi-\bar\delta} _{\theta=\bar\delta}dt\Big |&\le \e .\ea \] 
To finish the proof of the lemma it suffices to prove the following claim. 

\begin{claim}\label{claim-4} For every $\e>0$, there exists $\delta_0>0$ such that for each $0<\bar\delta \le \delta_0$ we can find  $t_0=t_0(\bar\delta)>0$ such that \[|J^{\bar\delta}(t)| \le \e\qquad \mbox{for}\,\,  t \ge t_0.\]   
\begin{proof}[Proof of Claim \ref{claim-4}] We prove the upper and lower bound separately. The proof of  the upper bound uses \eqref{eq-11} i.e. we  bound $J^{\bar \delta}(t)$ in terms of   the integral term and boundary term in  \eqref{eq-11}.
To bound the integral term, we use    \eqref{eq-13},  the curvature lower bound for $\alpha\in(1/2,1)$,  and the curvature upper bound for $\alpha \ge1$ (both shown in Proposition \ref{prop-curvbdd})   to obtain \[ \int_{\bar\delta}^{\pi-\bar\delta}- \fr{\alpha+1}{\alpha^2} \fr{p_t}{p^{\fr{2}{\alpha+1}}} d\theta \le  \int_{\bar\delta}^{\pi-\bar\delta}- \fr{\alpha+1}{\alpha^2}\fr{\ka^{\alpha-1}}{t}d\theta \to 0 \quad\mbox{as}\quad t\to \infty. \]
To bound the boundary term, we note that ${\displaystyle p^{\fr{\alpha-1}{\alpha+1}} p_\theta= (\alpha+1)\ka^{2\alpha-21}\ka_s= (\alpha+1)\fr{u_s}{u^{\fr{1}{\alpha}-1}}}$ and $\fr{1}{\alpha}-1 < 1$ for $\alpha>\fr{1}{2}$. Therefore, Proposition \ref{prop-us} and Theorem \ref{thm-curvaturedecay} imply that for any given $\e>0$ and $t_0>0$, there exists $\delta_0$ such that if $0<\bar \delta\le \delta_0$ and $t\ge t_0$ we have $\Big |\Big (\fr{\alpha+1}{\alpha} p^{\fr{\alpha-1}{\alpha+1}} p_\theta \Big )^{\theta=\pi-\bar\delta}_{\theta=\bar\delta} \Big |\le \e$. This completes the proof of the upper bound.

For the  lower bound, we will use   the 1-dim optimal Poincar\'e inequality, namely the bound 
\[\int_{\delta}^{\pi-\delta} f'(s)^2 \, ds -  \Big (\fr{\pi}{\pi-2\delta}\Big )^2 \int_{\delta}^{\pi-\delta} f(s)^2\, ds \ge 0\] 
which holds for every  smooth function $f$ with $f(\delta)=f(\pi-\delta)=0$. The equality holds for properly scaled sine functions. To apply it for our case, recall that   \[\fr{\alpha^2}{(\alpha+1)^2} J^{\bar\delta}(t) = \int_{\bar\delta}^{\pi-\bar\delta} \big ( (\ka^\alpha)_{\theta}^2- \ka^\alpha \big )  d\theta \]
and set   $U(\theta,t):=\ka^\alpha(\theta,t)$  and $-L(\theta,t):= \fr{U(\pi-\bar\delta,t)-U(\bar\delta,t)}{\pi-2\bar\delta} (\theta-\bar\delta) + U(\bar\delta,t)$ (note that we distinguish the notation of $U(\theta,t):=\ka^\alpha(\theta,t)$ from $u(n,t):=\bar \kappa^\alpha(n,t)$ 
which uses the geometric parametrization). Since $(U+L)(\bar \delta) =  (U+L)(\pi -\bar \delta)=0$,  the  Poincar\'e inequality above 
combined with  Young's inequality   imply 
\[\ba0&\le  \int_{\bar\delta}^{\pi-\bar\delta} \Big ( (U+L)^2_\theta- \Big (\fr{\pi}{\pi-2\delta}\Big )^2(U+L)^2 \Big ) \,d\theta \\
&= \int_{\bar\delta}^{\pi-\bar\delta} \Big ( U_\theta^2 - \fr{\pi^2}{(\pi-2\bar\delta)^2} U^2 + L_\theta^2-\fr{\pi^2}{(\pi-2\bar\delta)^2} L^2 + 2U_\theta \, L_\theta-  \fr{2\pi^2}{(\pi-2\bar\delta)^2} UL \Big )  \,d\theta\\
&\le   \int_{\bar\delta}^{\pi-\bar\delta} \Big ( U_\theta^2 - \fr{\pi^2}{(\pi-2\bar\delta)^2} U^2 + L_\theta^2-\fr{\pi^2}{(\pi-2\bar\delta)^2} L^2 \\ &\qquad \qquad \quad \,\,\,  + \fr{2\bar\delta}{\pi-2\bar\delta}U_\theta^2 +\fr{\pi-2\bar\delta}{2\bar\delta} L_\theta^2 + \fr{2\bar\delta \pi U^2}{(\pi-2\bar\delta)^2} + \fr{\pi^3L^2}{2\bar\delta(\pi-2\bar\delta)^2}  \Big ) \,d\theta  \\
&= \int_{\bar\delta}^{\pi-\bar\delta} \fr{\pi}{\pi-2\bar \delta} \big (U_\theta^2-U^2 \big )+ \fr{\pi}{2\bar\delta}\, \big (L_\theta^2+\fr{\pi}{\pi-2\bar\delta}L^2 \big )\,d\theta \ea\]
We conclude that
$$J^{\bar \delta}(t) : =   \frac{(\alpha+1)^2}{\alpha^2}   \int_{\bar\delta}^{\pi-\bar\delta} \big (  U_{\theta}^2- U \, \big )  d\theta  \geq 
 \frac{(\alpha+1)^2}{\alpha^2}   \fr{\pi-2\bar\delta}{\pi}\fr{\pi}{2\bar\delta} 
\int_{\bar\delta}^{\pi-\bar\delta}  \big ( L_\theta^2+\fr{\pi}{\pi-2\bar\delta}L^2 \big ) d\theta.$$
To estimate the last integral above we observe that by  Theorem \ref{thm-curvaturedecay},  we have $|L|$ and $|L_\theta| \le C(M_0)\bar\delta^{2/3}$ on $[\bar\delta,\pi-\bar\delta]$ for  all  $\bar\delta \in (0, \fr{\pi}{4})$ and $t>3$.  Hence, we have
 \[J^{\bar\delta}(t)  \ge -C(M_0,\alpha)\,  \bar\delta^{2\fr{2}{3}-1} =-C(M_0,\alpha)\, \bar\delta^{\fr{1}{3}}\] 
 which gives the bound from below.  This completes the proof of the claim. 
\end{proof}
\end{claim}
\end{proof}

\end{lemma}

We are now  in position to give he  proof of our main convergence  result, Theorem  \ref{thm-goal}. We have already observed in section 2 that Theorem \ref{thm-main}  follows from Theorem  \ref{thm-goal}. 

 \begin{proof}[Proof of Theorem \ref{thm-goal}]
Recall $U(\theta,t): =\ka^\alpha (\theta,t)$ solves the equation 
\be\label{eqn-uuu} 
U_t=\alpha U^{1+\frac{1}{\alpha}}(U_{\theta\theta}+U)\qquad \mbox{ on} \,\,  (0,\pi)\times (0,\infty).
\ee
For a  given time sequence $t_i \to \infty$, we define   the sequence of solutions $U^i(\theta,t):= U(\theta,t+t_i)$.
By Proposition \ref{prop-curvbdd}, the sequence $\{ U^i \}$  is locally uniformly bounded from above and below in spacetime and $i \gg1$.  That is,  for any compact spacetime region, there is  $i_0 \gg 1$ such  that  
$\{ U^i \}_{i \geq i_0}$ is  uniformly bounded from above and below by positive numbers. 
This implies that equation \eqref{eqn-uuu} is  uniformly parabolic for $U=U^i$, $i \geq i_0$ and therefore parabolic regularity theory implies that we  have locally uniform control on derivatives of the $u_i$ of all orders. 
 By the Arzel\`a-Ascoli theorem, we can find a subsequence,  still denoted by $U^i$, such that $U^i \to \bar U$
 uniformly on compact sets but also  $$U^i \to \bar U \qquad \text{ in }\,\,  C^\infty_{loc}((0,\infty)\times(-\infty,\infty)).$$
 Then, the Lemma \ref{lemma-main} implies that $\bar U_{\theta\theta}+\bar U=0$,  thus $\p_t \bar U=0$. In addition, Proposition \ref{prop-curvbdd} and Theorem \ref{thm-curvaturedecay} give $\bar U > 0$ and $\displaystyle \lim_{\theta \to 0} \bar U(\theta)=\lim_{\theta \to \pi} \bar U(\theta)=0$. Hence, we have 
$$\bar U(\theta)=c\, \sin \theta$$ for some constant $c>0$. We will next show that $c=m(\alpha)$, where $m(\alpha)$ is given by \eqref{eq-speed}. For this, it suffices to show that 
\be\label{eqn-upi}
U(\pi/2,t):=\ka^\alpha(\pi/2 ,t ) \to m(\alpha), \qquad \mbox{as} \,\, t\to \infty.
\ee

\noindent{\em Proof of \eqref{eqn-upi}:} Let's suppose first that $\liminf_{t\to \infty}  U(\pi/2,t)< m(\alpha)$. Then in view of the curvature lower bound in Proposition \ref{prop-curvbdd}, there is a sequence $t_i \to \infty$ such that $U(\theta,t_i)\to m'\sin \theta $ locally smoothly on $(0,\pi)$ for some $m'\in(0,m(\alpha))$. {Let $\big (x_1(\theta,t), x_2(\theta,t)\big )$ be the position vector of our solution $M_t$ parametrized by $\theta$}. For small $\e>0$,  this convergence and  \eqref{eq-curveintheta} imply that  we have, for $x_1(\theta,t)$,\begin{align}\label{eq-x_1conv}
x_1(\pi-\e,t)-x_1(\e,t)= \int_{\e}^{\pi-\e} \fr{\sin \theta}{\ka(\theta,t)}d\theta \to (m')^{-1/\alpha}\int_\e^{\pi-\e} (\sin \theta)^{1-\fr{1}{\alpha}} d\theta\qquad \mbox{as}\,\,  t\to \infty.
\end{align} 
Recall the assumptions  of Theorem \ref{thm-goal} and Proposition \ref{prop-basic} which imply that $M_{t_0}$ is a graph on $(-1,1)$, an interval of length $2$.  
In view of \eqref{eq-speed} and $m' < m(\alpha)$, we can find a small $\e(m')>0$ depending on $m'$ and a large $t_0(\e,m')>0$ depending on $\e,m'$ such that $x_1(\pi-\e,t_0)-x_1(\e,t_0)>2$. This gives a contradiction.  Therefore, 
\be\label{eq-liminf u tip}
\liminf_{t\to \infty}  U(\pi/2,t) \geq m(\alpha).
\ee

Next, suppose $\limsup_{t\to \infty}  U(\pi/2,t) > m(\alpha)$ and hence there is a sequence $t_i\to \infty$ such that $U(\pi/2,t_i) \ge (1+4\e)\, m(\alpha)$,  for some $\e>0$. In view of the Harnack estimate Proposition \ref{prop-harnack}, there is $c(\e)>0$ such that $U(\pi/2,t) \ge (1+3\e)\, m(\alpha)$ for $t\in[t_i,(1+c)t_i]$. Meanwhile, the inequality \eqref{eq-liminf u tip} implies that there is $\bar t>0$ such that $U(\pi/2,t) > (1-c\e)\, m(\alpha) $ for $t>\bar t$. Note that $\p_tx_2(\pi/2,t)= \ka^\alpha(\pi/2,t)$ and therefore, 
\[\ba x_2(\pi/2,(1+c)\,  t_i) &= x_2(\pi/2,t_i) + \int_{t_i}^{(1+c)t_i} \ka^\alpha(\pi/2,\tau)\,  d\tau \\
&\ge [(1-c\e)\, m(\alpha)\, t_i -C] + (1+3\e)\, (c t_i)\, m(\alpha)\\
&= m(\alpha) \, \Big (1+\fr{2c\e}{1+c}\Big ) (1+c)\, t_i -C.  \ea\]
On the other hand, we can put a translating soliton of speed $m(\alpha)\, \Big (1+\fr{c\e}{1+c}\Big )$ above $M_0$ and inside $\{|x_1|< 1-\delta\}$,  for some $\delta(\e,c)>0$ depending on $\e,c$ at the initial time $t=0$. Then, by the  comparison principle 
\[\ba x_2(\pi/2,(1+c)\, t_i) \le m(\alpha)\, \Big (1+\fr{c\e}{1+c}\Big ) (1+c)\, t_i + C\ea\]
which contradicts  the previous inequality for $t_i \gg 1$.  This completes the proof of  \eqref{eqn-upi}. 

\medskip
We have just seen that  the sequence $U^i$ smoothly converges to  $\bar U = m(\alpha)\, \sin \theta$ on compact sets along arbitrary sequence. Thus,  $U(\cdot,t)\to \bar U$  in $C^\infty_{loc}((0,\pi))$ as $t\to\infty$. From the convergence \eqref{eq-x_1conv} with $m'=m(\alpha)$ and Proposition \ref{prop-basic}, it is easy to see $x_1(\pi/2,t)$, the $x_1$ coordinate of the tip, converges to $0$ as $t\to\infty$. Then \eqref{eq-curveintheta}, Proposition \ref{prop-translator}  and the convergence of
  $\ka(\theta,t)$ to $\big ( m(\alpha) \,  \sin \theta \big )^{1/\alpha}$ yield our desired convergence of the graphical function stated in Theorem \ref{thm-goal}. This completes  the proof of Theorem \ref{thm-goal}.

\end{proof}

%{\color{blue}
%
% Up to a translation of the origin, a $C^2$-strictly convex curve is completely determined by $\ka(\theta)$ from the formula \be \label{eq-formula}x_1(\theta_2)-x_1(\theta_1)= \int_{\theta_1}^{\theta_2} \fr{\sin \theta}{\ka(\theta)}d\theta \quad\mbox{and}\quad x_2(\theta_2)-x_2(\theta_1)= \int_{\theta_1}^{\theta_2} -\fr{\cos \theta}{\ka(\theta)}d\theta \ee as appear in Lemma 4.1.1 of \cite{GH} (Therein, $\theta$  is the angle of tangent vector and differ by $\pi/2$ with ours). In this regards, the translators are characterized by $\ka^\alpha(\theta) =c \sin (\theta+\theta_0)$ for $\theta \in (-\theta_0,\pi-\theta_0)$ where $c>0$ and $\theta_0$ are determined by the width of the slab and the direction of translating. By the the formula, it could be checked that the width $d>0$ and $c>0$ are actually related by $c^{1/ \alpha} d= {\int_0^\pi \sin^{1-\fr{1}{\alpha}}\theta d\theta} $ (and this indicates the formula also could be used to show there is no translator in a slab for $\alpha <1/2$ and it should be an entire graph). Without loss of generality, we may assume the slab is parallel to $x_2$-axis with the width $d=\int_0^\pi \sin^{1-\fr{1}{\alpha}}\theta d\theta$. By \eqref{eq-formula} and the assumptions, it suffices to show \[\ka^\alpha(\theta,t) \xrightarrow{C^\infty_{\text{loc}}((0,\pi))}  \sin \theta\quad \mbox{as}\,\, t\to \infty  \] for the proof of the theorem.
%
%
%}

%%%%%%%%%%%%%%%%%%%%%%%%%%%%%%%%%%%%%%%%%%%%%%%%%%%%%%%%%%%%%%%%%%%%%%%%%%%%%%%%%%%%%%%%%%%%%%%%%%%%%%
%------------------------------------------------------------------------------%

\bigskip
\bigskip
\bigskip

\centerline{\bf Acknowledgements}

\smallskip 

\noindent K. Choi has been partially supported by NSF grant DMS-1811267.

\noindent P. Daskalopoulos and B. Choi have been partially supported by NSF grant DMS-1600658.

\bigskip
\bigskip

\end{document}